\documentclass[12pt,a4paper,reqno]{amsart}
\usepackage{amsmath,amsthm,amssymb,mathtools,enumitem,setspace,multicol}
\usepackage{a4wide}
\usepackage[margin=2cm]{geometry}
\usepackage{etoolbox}

\usepackage[utf8]{inputenc}

\usepackage[pagebackref=true,colorlinks=true,citecolor=blue,linkcolor=red,urlcolor=black]{hyperref}

\makeatletter
\patchcmd{\@setaddresses}{\indent}{\noindent}{}{}
\patchcmd{\@setaddresses}{\indent}{\noindent}{}{}
\patchcmd{\@setaddresses}{\indent}{\noindent}{}{}
\patchcmd{\@setaddresses}{\indent}{\noindent}{}{}
\makeatother

\title[The non-commuting, non-generating graph of a finite simple group]{The non-commuting, non-generating graph\\of a finite simple group}

\author{Saul D. Freedman}
\subjclass[2020]{Primary: 20E32; Secondary: 20D60, 05C25}
\thanks{\textit{Keywords}: simple groups; non-commuting non-generating graph; non-generating graph; generating graph; graphs defined on groups}
\address{\parbox{\linewidth}{
School of Mathematics and Statistics, University of St Andrews, St Andrews, KY16 9SS, UK\\
\textit{Current address}: School of Mathematics, Monash University, Clayton, VIC 3800, Australia}\vspace{.2cm}}
\email{saul.freedman1@monash.edu}

\newtheorem{thm}{Theorem}[section]
\newtheorem*{thm*}{Theorem}
\newtheorem*{conj*}{Conjecture}
\newtheorem{cor}[thm]{Corollary}

\newtheorem{lem}[thm]{Lemma}

\newtheorem{prop}[thm]{Proposition}

\newtheorem*{rem*}{Remark}
\theoremstyle{remark}

\theoremstyle{definition}
\newtheorem{rem}[thm]{Remark}
\newtheorem{defn}[thm]{Definition}
\newcounter{claim}[thm]

\numberwithin{equation}{section}

\makeatletter
\newcommand{\slteq}{%
  \mathrel{\mathpalette\sl@unlhd\relax}%
}

\newcommand{\sl@unlhd}[2]{%
  \sbox\z@{$#1\lhd$}%
  \sbox\tw@{$#1\leqslant$}%
  \dimen@=\ht\tw@
  \advance\dimen@-\ht\z@
  \ifx#1\displaystyle
    \advance\dimen@ .2pt
  \else
    \ifx#1\textstyle
      \advance\dimen@ .2pt
    \fi
  \fi
  \ooalign{\raisebox{\dimen@}{$\m@th#1\lhd$}\cr$\m@th#1\leqslant$\cr}%
}
\makeatother

\renewcommand{\le}{\leqslant}
\renewcommand{\ge}{\geqslant}
\renewcommand{\trianglelefteq}{\slteq}

\newcommand\diam{\mathrm{diam}}
\newcommand\sd{\mkern1.5mu{:}\mkern1.5mu}
\newcommand{\nonsplit}[2]{#1\raisebox{0.6ex}{$\cdot$} #2}
\newcommand\nc{\Xi}
\newcommand\nd{\Xi^+}
\newcommand\nongen{\overline{\Gamma}(G)}
\newcommand\cntr{\mathcal{Z}}

\begin{document}

\onehalfspacing

\begin{abstract}
Let $G$ be a group such that $G/Z(G)$ is finite and simple. The \emph{non-commuting, non-generating graph} $\nc(G)$ of $G$ has vertex set $G \setminus Z(G)$, with edges corresponding to pairs of elements that do not commute and do not generate $G$. Complementing our previous investigation of this graph for non-simple groups, we show that $\nc(G)$ is connected with diameter at most $5$, with smaller upper bounds for certain families of groups. Using these bounds, we then prove that when $G$ is simple, the diameter of the complement of the generating graph of $G$ has a tight upper bound of $4$, with the exception of at most one group with a graph of diameter $5$.
\end{abstract}

\maketitle

\section{Introduction}
\label{sec:intro}

The generating graph of a group $G$ has been investigated by many authors (e.g., \cite{bgk,burnessspread,crestani}), and encodes information about the $2$-generation properties of $G$: its vertex set is $G \setminus \{1\}$ and its edges are the generating pairs for $G$. The complement of this graph has also been studied (e.g., in \cite{lucchininongen}), and for $G$ non-abelian, its edge set necessarily includes all pairs of commuting elements. The remaining (and more interesting) non-generating pairs are described by the \emph{non-commuting, non-generating graph} $\nc(G)$ of $G$, with vertex set $G \setminus Z(G)$ and $\{x,y\}$ an edge if and only if $[x,y] \ne 1$ and $\langle x, y \rangle \ne G$.

Since its introduction in \cite{nilppaper}, the study of the graph $\nc(G)$ has been motivated by the importance of the $2$-generation properties of groups throughout (abstract and computational) group theory and its applications, and by Cameron's \cite[\S2.6]{camerongraphs} \emph{hierarchy of graphs}. This is a sequence of graphs defined on $G$ (for convenience, we will take the vertex set of each to be $G \setminus \{1\}$) such that, as long $G$ is non-abelian, each is a spanning subgraph of the next graph in the sequence. The final three of these graphs are the commuting graph (whose edges are the commuting pairs of elements of $G$), the non-generating graph (i.e., the complement of the generating graph), and the complete graph. The well-studied generating graph is the difference between the final two graphs. It is therefore natural to investigate the next difference, between the non-generating graph and the commuting graph, with all vertices corresponding to central elements deleted (these would otherwise necessarily be isolated). This difference is precisely $\nc(G)$.

It was proved in \cite{nilppaper} that if $G$ is nilpotent (or more generally, if all maximal subgroups of $G$ are normal) and $\nc(G)$ has an edge, then either the subgraph $\nd(G)$ induced by the non-isolated vertices of $\nc(G)$ is connected with diameter $2$, or $\nc(G)$ itself is connected with diameter $3$. In our companion paper \cite{nonsimplepaper}, we extend these results to the general case where $G/Z(G)$ is not simple. The only additional possibilities here are that $G$ is infinite and $\nd(G)$ has diameter $3$ or $4$, or that $\nc(G)$ is the union of two connected components of diameter $2$ (we provide in that paper a detailed structural description of the finite groups $G$ satisfying this case).

In this paper, we address the case where $G/Z(G)$ is finite and (non-abelian) simple. Our main theorem assumes that $G$ itself is simple. Here, $\diam(\nc(G))$ denotes the diameter of $\nc(G)$, which is always at least $2$, since no vertex is adjacent to any of its powers. Compared to Theorem 6.1.4 from the thesis \cite{saulthesis} on which this paper is based, our theorem takes into account new information on the maximal subgroups of the monster group $\mathbb{M}$; see Remark~\ref{rem:monsteroddmax}.

\begin{thm}
\label{thm:ncsimple}
Let $G$ be a non-abelian finite simple group.
\begin{enumerate}[label={(\roman*)},font=\upshape]
\item \label{ncsimple1} $\nc(G)$ is connected with diameter at most $5$.
\item \label{ncsimple2} If $G = \mathbb{M}$, then $\diam(\nc(G)) \in \{4,5\}$. If instead $G$ is any other sporadic simple group or the Tits group, then $\diam(\nc(G)) \le 4$.
\item \label{ncsimple3} For certain simple groups $G$, Table~\ref{table:simplediams} lists exact values or upper bounds for $\diam(\nc(G))$.
\end{enumerate}
\end{thm}

\begin{table}[ht]
\centering
\renewcommand{\arraystretch}{1.1}
\caption{Exact values or upper bounds for $\diam(\nc(G))$, for certain simple groups $G$.}
\label{table:simplediams}
\begin{tabular}{ ccc }
\hline
$G$ & Conditions & $\diam(\nc(G))$\\
\hline
$\mathrm{M}_{11}, \mathrm{M}_{12}, \mathrm{M}_{22}, \mathrm{J}_{2}$ & & $2$\\
$\mathrm{PSL}(2,q)$ & $q$ even or $q \le 9$ & $2$\\
$\mathrm{M}_{23}, \mathrm{J_1}$ & & $3$\\
$\mathrm{PSL}(2,q)$ & $q$ odd and $q \ge 11$ & $3$\\
$\mathbb{B}, \mathrm{PSU}(7,2)$ & & $4$\\
$A_n$ & $n$ even & $\le 3$\\
$A_n$ & $n$ odd & $\le 4$\\
$\mathrm{PSL}(n,q), \mathrm{Sz}(q)$ & & $\le 4$\\
$G_2(q), {}^2G_2(q), {}^3D_4(q), F_4(q), E_8(q)$ & $q$ odd & $\le 4$\\
\hline
\end{tabular}
\end{table}

\begin{rem}
\label{rem:simplediams}
Table~\ref{table:simplediams} shows that there exist non-abelian finite simple groups $G$ with $\diam(\nc(G))$ equal to $2$, $3$ or $4$. We note that the baby monster group $\mathbb{B}$, the monster group $\mathbb{M}$ and $\mathrm{PSU}(7,2)$ are the only known simple groups $G$ for which $\diam(\nc(G)) > 3$. Indeed, we do not know if infinitely many such groups exist, nor if $\diam(\nc(G)) = 5$ is possible.
\end{rem}

At the end of each of Sections~\ref{subsec:altnc}, \ref{subsec:exceptionalnc}, \ref{subsec:nclin} and \ref{subsec:unitarync}, we provide further examples of diameters achieved by various simple groups, computed using Magma \cite{magma} (see \S\ref{subsec:prelims} for a discussion of our computational methods). Notably, none of our examples meet the upper bounds in rows 6--9 of Table~\ref{table:simplediams}.

Continuing the comparison between $\nc(G)$ and the generating graph of $G$ from \cite{nilppaper,nonsimplepaper}, we note that in the non-abelian finite simple case, the diameter of the former graph may be larger than that of the latter, which is equal to $2$ by \cite[Theorem 1.2]{bgk}.

\begin{rem}
It is easy to determine the connectedness and diameter of the non-commuting, non-generating graphs of certain infinite simple groups $G$. For instance, if $G$ is a \emph{Tarski monster}, then each nontrivial proper subgroup of $G$ has order a fixed prime, and so each vertex of $\nc(G)$ is isolated. Additionally, it is a straightforward consequence of \cite[Proposition 2.1]{abdollahi} that $\diam(\nc(G)) = 2$ whenever $G$ is not $2$-generated. However, it is an open problem to investigate the connectedness and diameter of $\nc(G)$ when the infinite simple group $G$ is $2$-generated and $\nc(G)$ has an edge.
\end{rem}

\begin{rem}
Our proof of Theorem~\ref{thm:ncsimple} involves Proposition~\ref{prop:intmonst}, which corrects \cite[Theorem 1.1]{intgraph} on the \emph{intersection graph} (defined in Section~\ref{subsec:prelims}), in the case of the monster group. Additionally, our proof uses Theorem~\ref{thm:unisubspaces}, on subspace stabilisers in simple unitary groups. In \cite[Ch.~3]{saulthesis}, we prove more general versions of that theorem, to obtain lower bounds for the base sizes of primitive subspace actions of almost simple classical groups.
\end{rem}

We also note that our proof of Theorem~\ref{thm:ncsimple}\ref{ncsimple1} relies on the classification of finite simple groups. In \S\ref{subsec:maxsubs}, we prove the following special case that does not depend on the classification.

\begin{prop}
\label{prop:ncsimpleeven}
Let $G$ be a non-abelian finite simple group, and suppose that every maximal subgroup of $G$ has even order. Then $\nc(G)$ is connected with diameter at most $5$.
\end{prop}

In fact, it follows from the classification (see Theorem~\ref{thm:simpleoddmax} and Remark~\ref{rem:monsteroddmax} below) that in addition to a few sporadic groups, the finite simple groups containing maximal subgroups of odd order belong to only three infinite families. Thus the above proposition serves as a useful reduction for the proof of Theorem~\ref{thm:ncsimple}\ref{ncsimple1}.

Now, if $G/Z(G)$ is simple, then $Z(G)x$ is a vertex of $\nc(G/Z(G))$ whenever $x \in G \setminus Z(G)$. Thus \cite[Lemma 2.9(ii)]{nonsimplepaper} yields the following corollary of Theorem~\ref{thm:ncsimple}\ref{ncsimple1}.

\begin{cor}
\label{cor:simplecentext}
Let $G$ be a group such that $G/Z(G)$ is finite and simple. Then $\nc(G)$ is connected, with $\diam(\nc(G)) \le \diam(\nc(G/Z(G))) \le 5$.
\end{cor}

If instead $G/Z(G)$ is finite and almost simple but not simple, then we observe from \cite[Theorems 1.2--1.3 \& Lemma 6.5]{nonsimplepaper} that $\nc(G)$ is connected with diameter $2$ or $3$. Magma computations show that $\diam(\nc(S_5)) = 2$ and $\diam(\nc(S_6)) = 3$, and so both possibilities occur.

Finally, as a consequence of Theorem~\ref{thm:ncsimple}, we will prove the following analogue of that theorem for the non-generating graph $\nongen$ of a non-abelian finite simple group $G$. Recall that this graph is the complement of the generating graph of $G$. Hence its vertex set is $G \setminus \{1\}$, and vertices $x$ and $y$ are adjacent if and only if $\langle x, y \rangle \ne G$. Note that $\nc(G)$ is a spanning subgraph of $\nongen$ (since $Z(G) = 1$). As above, we write $\mathbb{B}$ to denote the baby monster group and $\mathbb{M}$ to denote the monster group. Similarly to Theorem~\ref{thm:ncsimple}, we incorporate new information on the maximal subgroups of $\mathbb{M}$ (compared to \cite[Theorem 6.5.4]{saulthesis}) and its intersection graph.

\begin{thm}
\label{thm:nongensimple}
Let $G$ be a non-abelian finite simple group. The non-generating graph $\nongen$ is connected with diameter $4$ or $5$ if $G = \mathbb{M}$, at most $4$ if $G \not\cong \mathbb{M}$, and at most $3$ if every maximal subgroup of $G$ has even order. Furthermore, if $G \in \{\mathbb{B}, \mathrm{PSU}(7,2)\}$, then $\diam(\nongen) = 4$.
\end{thm}

An earlier version of this theorem was communicated by this paper's author to Peter Cameron, and appears as a remark after Proposition 12.3 in his paper \cite{camerongraphs}. That version of the theorem is a key component of the proof of \cite[Theorem 1]{lucchininongen} (see the proof of \cite[Proposition 10]{lucchininongen}), which investigates the connectedness and diameter of $\nongen$ for a finite group $G$. This graph is closely related to the \emph{soluble graph} of $G$, recently studied in \cite{solgraph}; see Remark 7 in that paper.

The remainder of this paper is structured as follows. The focus of \S\ref{sec:prelimsandmax} is on preliminary results, and includes a comment on our computational methods and a proof of Proposition~\ref{prop:ncsimpleeven}. In \S\ref{sec:altsporexcep}, we prove the alternating, sporadic and exceptional cases of Theorem~\ref{thm:ncsimple}. We consider matrices and stabilisers of subspaces in linear and unitary groups in \S\ref{sec:matprel}, and then apply our results in \S\ref{sec:nclinearuni} to prove Theorem~\ref{thm:ncsimple} for those groups. Additionally, \S\ref{sec:nclinearuni} concludes with a proof of Theorem~\ref{thm:nongensimple}.

\section{Preliminaries}
\label{sec:prelimsandmax}

Here, we present preliminary results on the non-commuting, non-generating graph of a group, and on maximal subgroups of finite simple groups; in particular, we prove Proposition~\ref{prop:ncsimpleeven}.

\subsection{The non-commuting, non-generating graph}
\label{subsec:prelims}

In this subsection, $G$ is an arbitrary group, unless specified otherwise. Given vertices $x$ and $y$ of a graph $\Gamma$, we denote the distance in $\Gamma$ between $x$ and $y$ by $d(x,y)$, and if $x$ and $y$ are adjacent, then we write $x \sim y$.

\begin{lem}[{\cite[Corollary 6]{nilppaper}}]
\label{lem:propernc}
Let $H$ be a proper non-abelian subgroup of $G$. Then the induced subgraph of $\nc(G)$ corresponding to $H \setminus Z(H)$ is connected with diameter $2$.
\end{lem}

\begin{prop}
\label{prop:isolvertnc}
Suppose that $G$ is finite and insoluble. Then $\nc(G)$ has no isolated vertices. Equivalently, each element of $G \setminus Z(G)$ is a non-central element of some maximal subgroup.
\end{prop}

\begin{proof}
It is clear from the definition of $\nc(G)$ that a vertex $x \in \nc(G)$ is isolated if and only if $x$ lies in a unique maximal subgroup $M$ of $G$ and $x \in Z(M)$. Since $G$ is finite, it follows from \cite[Proposition 2.6]{nonsimplepaper} that each isolated vertex of $\nc(G)$ lies in an abelian maximal subgroup of $G$. However, each maximal subgroup of the finite insoluble group $G$ is non-abelian \cite{herstein}.
\end{proof}

Our next result highlights a connection between $\nc(G)$ and the \emph{intersection graph} $\Delta_G$ of $G$. This latter graph, introduced in \cite{csakany}, has vertex set the proper nontrivial subgroups of $G$, with subgroups adjacent if and only if they intersect nontrivially.

\begin{lem}
\label{lem:intncgraphs}
Let $\Gamma$ be equal to $\nc(G)$ or the non-generating graph of $G$. Additionally, suppose that $Z(G) = 1$, and that $\Gamma$ is connected with finite diameter. Then $\diam(\Gamma) \ge \diam(\Delta_G)-1$.
\end{lem}

\begin{proof}
By Propositions 12.1 and 2.3 of \cite{camerongraphs} (which assume that $G$ is finite, but whose proofs also hold for infinite groups), the diameters of $\Delta_G$ and the non-generating graph of $G$ differ by at most one. As $Z(G)$ is trivial, $\nc(G)$ is a spanning subgraph of the non-generating graph of $G$, and so the diameter of the former is at least the diameter of the latter.
\end{proof}

We conclude this subsection with a a brief description of our computational methods for determining $\diam(\nc(G))$ in Magma (for sufficiently small $G$). Let $X$ be a set of representatives of generators for the non-central cyclic subgroups of $G$, and for each $x \in X$, let $Y_x$ be a set of representatives for the elements of $X$ up to conjugacy in $C_G(x)$. We also obtain a set $X'$ from $X$ by identifying elements if they generate conjugate cyclic subgroups. It is easy to see that $\mathrm{Aut}(\nc(G))$ contains $\mathrm{Aut}(G)$, and that $\diam(\nc(G))= \max\{2,\diam(\Gamma)\}$, where $\Gamma$ is the graph obtained from $\nc(G)$ by identifying vertices that generate identical cyclic subgroups. We obtain $\diam(\nc(G))$ by calculating the maximum distance in $\Gamma$ between the elements in each pair $(x,y)$ with $x \in X'$ and $y \in Y_x$, via computations involving centralisers of elements and intersections of maximal subgroups.

The Magma code used in these calculations is contained in the file \texttt{diam\char`_nc\char`_ng} in \cite{phdcode}; see also \cite[Appendix A]{saulthesis} for further information. When using this code to calculate $\diam(\nc(G))$ in this paper, we always represent $G$ as a permutation group of minimal degree. We note that the time and memory requirements for running this code are negligible for sufficiently small groups (such as $A_5$), but are in the order of a few days and a few gigabytes for the largest groups investigated in this paper via these methods, namely $G_2(3)$, $\mathrm{PSL}(4,3)$ and $\mathrm{M}_{23}$.

\subsection{Maximal subgroups of non-abelian finite simple groups}
\label{subsec:maxsubs}

Assume throughout this subsection that $G$ is non-abelian, finite and simple.

\begin{lem}
\label{lem:simplenoncent}
Let $L$ and $M$ be maximal subgroups of $G$, with $|L|$ even.
\begin{enumerate}[label={(\roman*)},font=\upshape]
\item \label{simplenoncent1} If $Z(L)$ contains an involution $a$, then $L \setminus Z(L)$ contains a $G$-conjugate of $a$.
\item \label{simplenoncent2} $L$ contains an involution that does not lie in $Z(L) \cup Z(M)$.
\item \label{simplenoncent3} Suppose that $|M|$ is even, and let $a$ be an involution of $L \setminus M$. Then $M \setminus Z(M)$ contains an involution $b$ that does not commute with $a$.
\item \label{simplenoncent4} Suppose 
that $L \cap M$ lies in $Z(M)$ and contains an involution $f$. Additionally, let ${u \in L \setminus M}$. Then $L \setminus M$ contains an involution $f'$ with $[u,f'] \ne 1$.
\end{enumerate}
\end{lem}

\begin{proof}\leavevmode

\noindent (i) Suppose that $Z(L)$ contains an involution $a$, so that $C_G(a) = L$. By \cite[Corollary 1]{glauberman}, there exists $g \in G$ such that $a \ne a^g \in C_G(a) = L$. In fact, $g \notin L$, as $a \in Z(L)$. Thus $L \ne L^g = C_G(a^g)$, and so the involution $a^g$ is non-central in $L$.

\medskip

\noindent (ii) We will prove that $L$ contains an involution $s \notin Z(L) \cup Z(M)$. By (i), there exists an involution $y \in L \setminus Z(L)$. If $y \notin Z(M)$, then we can set $s = y$. If instead $y \in Z(M)$ (so that $L \ne M$), then $C_G(y) = M$. Additionally, since $Z(L)$ and $L \cap M$ are proper subgroups of $L$, there exists $h \in L \setminus (Z(L) \cup M)$, and $y^h$ is a non-central involution of $L$. Moreover, $M \ne M^h = C_G(y^h)$. Therefore, $y^h \notin Z(M)$, and so we can set $s = y^h$.

\medskip

\noindent (iii) Suppose first that $|Z(M)|$ is odd, let $S$ be the set of involutions of $M$, and let $Q_M \trianglelefteq M$ be the subgroup generated by $S$. If the involution $a \in L \setminus M$ commutes with each element of $S$, then $a \in C_G(Q_M)$, and so $Q_M \trianglelefteq \langle M, a \rangle = G$, contradicting the simplicity of $G$. Thus there exists $r \in S \setminus C_G(a)$. As $S \cap Z(M) = 1$, we can set $b = r$.

If instead $Z(M)$ contains an involution $z$, then by applying (i) to $M$, we deduce that $M \setminus Z(M)$ contains an involution $c$. If $[a,c] \ne 1$, then we can set $b = c$. Otherwise, as $a \notin M = C_G(z)$, we see that $[a,zc] = [a,z]^c \ne 1$. Thus we may set $b$ to be the involution $zc \in M \setminus Z(M)$.

\medskip

\noindent (iv) Let $s$ be the involution from (ii), so that $s \in L$ and $s \notin L \cap M \le Z(M)$. If $[u,s] \ne 1$, then we can set $f' = s$. Assume therefore that $[u,s] = 1$. Since $C_G(f) = M$ and $C_G(f^s) = M^s \ne M$, we obtain $f^s \in L \setminus Z(M)$, and hence $f^s \notin L \cap M$. If $f^s \in C_G(u)$, then $C_L(u)$ contains $\langle s,f^s \rangle$, which contains $f$. This contradicts $u \notin M = C_G(f)$, and so we can set $f' = f^s$.
\end{proof}

The following lemma is one of the main ingredients in the proof of Theorem~\ref{thm:ncsimple}.

\begin{lem}
\label{lem:ncmaxsimple}
Let $L$ and $M$ be maximal subgroups of $G$ of even order, and let $x \in L \setminus Z(L)$ and $y \in M \setminus Z(M)$. Then $d(x,y) \le 5$, as vertices of $\nc(G)$. Moreover, if $L$ contains an involution $a$ such that $d(x,a) \le 1$, then $d(x,y) \le 4$.
\end{lem}

\begin{proof}
We will use Lemma~\ref{lem:propernc} several times in this proof without further reference. First, we may assume that $L \ne M$, as otherwise $d(x,y) \le 2$. Let $\mathcal{C}$ be the set of involutions in $L \setminus Z(L)$. Lemma~\ref{lem:simplenoncent}\ref{simplenoncent2} shows that there exists $c \in \mathcal{C} \setminus Z(M)$. Notice that $d(x,c) \le 2$. If $c \in M \setminus Z(M)$, then $d(c,y) \le 2$, and so $d(x,y) \le d(x,c)+d(c,y) \le 4$. If instead $c \in G \setminus M$, then Lemma~\ref{lem:simplenoncent}\ref{simplenoncent3} yields an involution $b \in M \setminus Z(M)$ such that $[c,b] \ne 1$. As $\langle c, b \rangle$ is dihedral, $c \sim b$. Since $d(b,y) \le 2$, we conclude in this case that $d(x,y) \le 5$, with $d(x,y) \le 4$ if $d(x,c) \le 1$.

Now, let $\mathcal{A}$ be the set of involutions $a \in \mathcal{C}$ with $d(x,a) \le 1$. By the previous paragraph, it remains to consider the case where $\mathcal{C} \subseteq (G \setminus M) \cup Z(M)$ and $\varnothing \ne \mathcal{A} \subseteq Z(M)$. Note that $c \in \mathcal{C} \setminus M$ in this case.

Let $a \in \mathcal{A} \subseteq Z(M)$, so that $C_G(a) = M$. Then $a \sim c$, and since $c \notin \mathcal{A}$, it follows from the definition of $\mathcal{A}$ that $a \ne x$ and $a \sim x$. Thus $x \notin M$. Lemma~\ref{lem:simplenoncent}\ref{simplenoncent4} (with $f = a$ and $u = x$) now shows that $L \cap M \not\le Z(M)$, as otherwise $L \setminus M$ would contain an involution in $\mathcal{A}$. Hence $L \cap M$ contains an element not centralising $M$, and the element $a$ not centralising $x$. As no group is the union of two proper subgroups, there exists $h \in L \cap M$ centralising neither $M$ nor $x$. Thus $d(h,y) \le 2$ and $x \sim h$, and therefore $d(x,y) \le 3$.
\end{proof}

\begin{proof}[Proof of Proposition~\ref{prop:ncsimpleeven}]
For each $x \in \nc(G)$, Proposition~\ref{prop:isolvertnc} gives $x \in K \setminus Z(K)$ for some maximal subgroup $K$ of $G$ (of even order). Thus $\diam(\nc(G)) \le 5$ by Lemma~\ref{lem:ncmaxsimple}.
\end{proof}

As the proofs of \cite[Corollary 1]{glauberman} and Proposition~\ref{prop:isolvertnc} are independent of the classification of finite simple groups, so are the proofs of Lemmas~\ref{lem:simplenoncent} and~\ref{lem:ncmaxsimple} and Proposition~\ref{prop:ncsimpleeven}.

To prove Theorem~\ref{thm:ncsimple}\ref{ncsimple1}, it remains to consider the non-abelian finite simple groups that contain maximal subgroups of odd order.

\begin{thm}[{\cite[Theorem 2]{liebecksaxl}}]
\label{thm:simpleoddmax}
Suppose that $G$ is not isomorphic to the monster group. Then $G$ has a maximal subgroup of odd order if and only if $G$ is isomorphic to one of:
\begin{enumerate}[label={(\roman*)},font=\upshape]
\item $A_p$, with $p$ prime, $p \equiv 3 \pmod 4$ and $p \notin \{7,11,23\}$;
\item $\mathrm{PSL}(n,q)$, with $n$ prime, $(n,q) \ne (3,4)$, and $q \equiv 3 \pmod 4$ if $n = 2$;
\item $\mathrm{PSU}(n,q)$, with $n$ an odd prime and $(n,q) \notin \{(3,3),(3,5),(5,2)\}$; or
\item the Mathieu group $\mathrm{M}_{23}$, the Thompson group $\mathrm{Th}$, or the baby monster group $\mathbb{B}$.
\end{enumerate}
\end{thm}

\begin{rem}
\label{rem:monsteroddmax}
Let $G$ be the monster group. By \cite[Theorem 2]{liebecksaxl}, any maximal subgroup of $G$ of odd order has shape $59 \sd 29$ or $71 \sd 35$. It was claimed in \cite{holmes59,holmes71} that maximal subgroups of $G$ isomorphic to $\mathrm{PSL}(2,59)$ and $\mathrm{PSL}(2,71)$ were constructed computationally (using data that were not made publicly available). This suggested that no maximal subgroup of shape $59 \sd 29$ or $71 \sd 35$ exists. However, a brand new computational paper \cite{dlpp59} shows (via publicly available code) that, up to conjugacy, $G$ has unique maximal subgroups of shape $59 \sd 29$ and $\mathrm{PSL}(2,71)$, but no subgroup isomorphic to $\mathrm{PSL}(2,59)$. Hence, up to conjugacy, $59 \sd 29$ is the unique maximal subgroup of $G$ of odd order.
\end{rem}

\section{Alternating, sporadic and exceptional simple groups}
\label{sec:altsporexcep}

In this section, we prove the alternating, sporadic and exceptional cases of Theorem~\ref{thm:ncsimple}. Note that we consider the Tits group ${}^2F_4(2)'$ together with the sporadic groups.

\subsection{Alternating groups}
\label{subsec:altnc}

Let $G$ be the alternating group $A_n$ of degree $n \ge 5$. We will prove the following result by considering the standard action of $G$ on the set $\Omega:=\{1,2,\ldots,n\}$.

\begin{thm}
\label{thm:ansnnc}
The graph $\nc(G)$ of the alternating group $G$ of degree $n \ge 5$ is connected with diameter at most $4$, or at most $3$ if $n$ is even.
\end{thm}

Recall that a \emph{derangement} of $G$ is an element with no fixed points.

\begin{lem}
\label{lem:derneighb}
Let $x$ be a derangement of $G$ such that $\langle x \rangle$ acts intransitively on $\Omega$. In addition, let $\{\alpha_1,\alpha_2\}$ and $\{\beta_1, \beta_2\}$ be $2$-subsets of distinct orbits of $x$ on $\Omega$. Finally, if $x$ has exactly two orbits on $\Omega$, then let $g:=(\alpha_1,\alpha_2)(\beta_1,\beta_2)$, and otherwise, let $g:=(\alpha_1,\alpha_2,\beta_1)$. Then $x \sim g$.
\end{lem}

\begin{proof}
If $\langle x \rangle$ has exactly two orbits on $\Omega$, then at least one of these orbits has length at least three. Hence, in each case, $x^g \ne x$, and so $[x,g] \ne 1$. Additionally, $\langle x, g \rangle$ acts intransitively on $\Omega$ and is therefore a proper subgroup of $G$. Thus $x \sim g$.
\end{proof}

\begin{lem}
\label{lem:altnonder}
Let $x, y \in G \setminus \{1\}$. Then $d(x,y) \le 4$. If $d(x,y) \ge 3$, then at least one of $x$ and $y$ is a derangement, and if $d(x,y) = 4$, then both are derangements.
\end{lem}

\begin{proof}
Let $\alpha$ and $\beta$ be distinct points in $\Omega$. The point stabiliser $G_\alpha \cong A_{n-1}$ is a maximal subgroup of $G$. As $Z(G_\alpha) = 1$, it follows from Lemma~\ref{lem:propernc} that $d(g,h) \le 2$ for all $g,h \in G_\alpha \setminus \{1\}$. Additionally, $H:=G_\alpha \cap G_\beta \cong A_{n-2}$ is maximal in each of $G_\alpha$ and $G_\beta$. Let $g \in G_\alpha$ and $k \in G_\beta$, such that $g, k \notin H$. Then $C_H(g) < H$, as otherwise $\langle H, g \rangle = G_\alpha$ would centralise $g$. Similarly, $C_H(k) < H$. Hence some element of $H$ centralises neither $g$ nor $k$, and it follows that $d(g,k) \le 2$.

To complete the proof, we will show that each derangement $x \in G$ is adjacent in $\nc(G)$ to some non-derangement. This follows immediately from Lemma~\ref{lem:derneighb} if $\langle x \rangle$ is intransitive on $\Omega$. Assume therefore that $x$ is an $n$-cycle $(\alpha_1, \ldots, \alpha_n)$, with $\alpha_i \in \Omega$ for each $i$, and $n$ is odd. It suffices to prove that there exists $t \in N_{G_{\alpha_1}}(\langle x \rangle)$ with $x^t \ne x$; it will then follow that $\langle x, t \rangle \le N_G(\langle x \rangle) < G$ and $[x,t] \ne 1$, so that $x \sim t$. For each $n$-cycle $r \in \langle x \rangle$ with $r \ne x$, there exists an element $s \in N_{(S_n)_{\alpha_1}}(\langle x \rangle)$ such that $x^s = r$. There are $\phi(n)-1 > 1$ of these $n$-cycles, where $\phi$ is Euler's totient function. We deduce that there exist $s, s' \in N_{(S_n)_{\alpha_1}}(\langle x \rangle)$ such that $x^s = x^{-1}$ and $x^{s'} = x^i$ for some $i \in \{2,\ldots,n-2\}$. Since $x^{ss'} = x^{-i} \ne x$, we can choose $t \in G \cap \{s,s',ss'\} \ne \varnothing$.
\end{proof}

We will write $\mathrm{supp}(g)$ to denote the support $\{\alpha \in \Omega \mid \alpha^g \ne \alpha\}$ of an element $g \in G$.

\begin{lem}
\label{lem:dermultorbits}
Let $x$ and $y$ be derangements of $G$ such that each of $\langle x \rangle$ and $\langle y \rangle$ acts intransitively on $\Omega$. Then $d(x,y) \le 3$.
\end{lem}

\begin{proof}
Let $\{\alpha_1,\alpha_2\}$ and $B:=\{\beta_1, \beta_2\}$ be $2$-subsets of distinct orbits of $\langle x \rangle$ on $\Omega$, such that $|\alpha_1^{\langle x \rangle}| \ge |\beta_1^{\langle x \rangle}|$, and let $\{\gamma_1,\gamma_2\}$ and $D:=\{\delta_1,\delta_2\}$ be $2$-subsets of distinct orbits of $\langle y \rangle$ on $\Omega$. We may assume that $\alpha_1 = \gamma_1$. If $n = 5$, then each of $\langle x \rangle$ and $\langle y \rangle$ has exactly two orbits on $\Omega$.

\medskip

\noindent \textbf{Case (a)}: $\langle x \rangle$ and $\langle y \rangle$ each have exactly two orbits on $\Omega$. We see from Lemma~\ref{lem:derneighb} that $x \sim a:=(\alpha_1,\alpha_2)(\beta_1,\beta_2)$ and $y \sim b:=(\alpha_1,\gamma_2)(\delta_1,\delta_2)$. In particular, if $a = b$, then $d(x,y) \le 2$.

Assume therefore that $a \ne b$, and hence either $\alpha_2 \ne \gamma_2$ or $B \ne D$. If $[a,b] = 1$, then either $\gamma_2 = \alpha_2$ and $B \cap D = \varnothing$; or $\{\gamma_2\} \cup D = \{\alpha_2\} \cup B$. As $|\alpha_1^{\langle x \rangle}| \ge 3$, we can repeat the argument with $\alpha_2$ replaced by an element of $\alpha_1^{\langle x \rangle} \setminus \{\alpha_1,\alpha_2\}$ to obtain $[a,b] \ne 1$. Since $\langle a, b \rangle$ is dihedral, $(x,a,b,y)$ is a path in $\nc(G)$ and $d(x,y) \le 3$.

\medskip

\noindent \textbf{Case (b)}: $\langle x \rangle$ and $\langle y \rangle$ each have at least three orbits on $\Omega$. By Lemma~\ref{lem:derneighb}, $x \sim a:=(\alpha_1,\alpha_2,\beta_1)$ and $y \sim b:=(\alpha_1,\gamma_2,\delta_1)$. Observe that $t:={\mathrm{supp}(a) \cup \mathrm{supp}(b)} \le 5 < n$, and so $\langle a, b \rangle < G$. Hence if $[a,b] \ne 1$, then $(x,a,b,y)$ is a path in $\nc(G)$, and so $d(x,y) \le 3$. If instead $[a,b] = 1$, then $t = 3$ and $b \in \{a,a^{-1}\}$, hence $a \sim y$ and $d(x,y) \le 2$.

\medskip

\noindent \textbf{Case (c)}: Exactly one of $\langle x \rangle$ and $\langle y \rangle$ has exactly two orbits on $\Omega$. We may assume that $\langle x \rangle$ has exactly two orbits on $\Omega$. Lemma~\ref{lem:derneighb} shows that $x \sim a:=(\alpha_1,\alpha_2)(\beta_1,\beta_2)$ and $y \sim b:=(\alpha_1,\gamma_2,\delta_1)$. Observe that $[a,b] \ne 1$ and $\langle a,b \rangle$ is intransitive. Hence $a \sim b$ and $d(x,y) \le 3$.
\end{proof}

\begin{proof}[Proof of Theorem~\ref{thm:ansnnc}]
Lemma~\ref{lem:altnonder} shows that $\nc(G)$ is connected with diameter at most $4$. Furthermore, if $n$ is even, then each derangement of $G$ generates a subgroup that acts intransitively on $\Omega$. Hence Lemmas~\ref{lem:altnonder} and~\ref{lem:dermultorbits} yield $\diam(\nc(G)) \le 3$.
\end{proof}

Magma calculations show that $\diam(\nc(A_n)) = 2$ whenever $5 \le n \le 10$. We do not know if there exists $n > 10$ such that $\diam(\nc(A_n)) > 2$.

\subsection{Sporadic groups}
\label{subsec:sporadicnc} Let $G$ be a sporadic finite simple group, or the Tits group ${}^2F_4(2)'$. In this subsection, we implicitly use \cite{ATLAS,wilsonsporadic} to determine information about maximal subgroups, conjugacy classes and centralisers in $G$. In particular, all maximal subgroups of the sporadic groups are listed in \cite{wilsonsporadic} (up to conjugacy), with the exceptions of the maximal subgroups of the monster group $\mathbb{M}$ isomorphic to $\mathrm{PGL}(2,13)$ and $\mathrm{Aut}(\mathrm{PSU}(3,4))$ constructed in \cite{dlp}, and the maximal subgroups isomorphic to $59 \sd 29$ constructed in \cite{dlpp59} (see the introductions of those last two papers for further details). Additionally, the subgroup $\mathrm{PSL}(2,59)$ listed in \cite{wilsonsporadic} is not in fact a subgroup of $\mathbb{M}$; see Remark~\ref{rem:monsteroddmax}.

\begin{thm}
\label{thm:spordiamnc}
The graph $\nc(G)$ of the sporadic group or Tits group $G$ is connected with diameter $4$ or $5$ if $G = \mathbb{M}$, and at most $4$ otherwise. If $G \in \{\mathrm{M}_{11},\mathrm{M}_{12},\mathrm{M}_{22},\mathrm{M}_{23},\mathrm{J}_1,\mathrm{J}_2,\mathbb{B}\}$, then $\diam(\nc(G))$ is given in Table~\ref{table:simplediams}.
\end{thm}

A few of the proofs in this subsection, where specified, involve straightforward computations in GAP \cite{GAP} that utilise the Character Table Library \cite{GAPchar}, with runtimes of at most a few seconds. This library contains the necessary information regarding all maximal subgroups of $G$, except for certain maximal subgroups of the monster group.

\begin{prop}
\label{prop:spormaxz}
Each maximal subgroup of $G$ has a centre of order at most $2$.
\end{prop}

\begin{proof}
For each maximal subgroup $M$ of $\mathbb{M}$ not included in the GAP Character Table Library, no element of $\mathbb{M}$ has a centraliser of order $|M|$, hence $Z(M) = 1$. For each other maximal subgroup of each $G$, we verify the result using GAP.
\end{proof}

In light of Lemma~\ref{lem:ncmaxsimple}, it will be useful to determine neighbours of involutions in $\nc(G)$.

\begin{prop}
\label{prop:sporinvol}
Let $x$ be an element of $G$ of order at least $3$, and suppose that $x$ lies in a maximal subgroup of $G$ of even order. Then $x$ is adjacent in $\nc(G)$ to some involution of $G$.
\end{prop}

\begin{proof}
It suffices to show that there exists a maximal subgroup of $G$ that contains both $x$ and an involution that does not centralise $x$. Let $M$ be a maximal subgroup of $G$ of even order, with $x \in M$. Since $|x| > 2$, Proposition~\ref{prop:spormaxz} shows that $x \notin Z(M)$. Let $T_M$ be the set of involutions of $M$. If $|T_M| > |C_M(x)|$, or if the normal subgroup $\langle T_M \rangle$ of $M$ is equal to $M$, then $x$ does not centralise every involution of $M$. Note that this second condition holds whenever $M$ is simple.

Using GAP, we see that if $x$ does not lie in any simple maximal subgroup, nor in any maximal subgroup $M$ with $|T_M| > |C_M(x)|$, then one of the following holds:
\begin{multicols}{2}
\begin{enumerate}[label={(\roman*)},font=\upshape]
\item $G = \mathrm{J}_1$ and $|x| \in \{7,19\}$;
\item $G = \mathrm{Ly}$ and $|x| \in \{37,67\}$;
\item $G = \mathrm{Co}_1$ and $|x| = 3$;
\item $G = \mathrm{J}_4$ and $|x| \in \{29,43\}$; or
\item $G = \mathrm{Fi}_{24}'$ and $|x| = 29$.
\end{enumerate}
\end{multicols}
Our computations here involve the fusion of $M$-conjugacy classes in $G$. As such, when $G = \mathbb{M}$, we consider only the set of maximal subgroups of $\mathbb{M}$ for which this fusion is stored in GAP (this is a subset of the set returned by the \texttt{NamesOfFusionSources} function). Similarly, when $G = \mathbb{B}$, our computations exclude the maximal subgroups of shape $(2^2 \times F_4(2)) \sd 2$, for which this fusion is not stored.

Now, in case (iii), $x$ is an element of the $G$-conjugacy class labelled $3\mathrm{A}$ in the \textsc{Atlas} \cite{ATLAS}, and lies in a maximal subgroup $K \cong (A_5 \times \mathrm{J}_2) \sd 2$, which clearly satisfies $\langle T_K \rangle = K$. In all other cases, $|x|$ is odd and $C_G(x) = \langle x \rangle$, hence $x$ does not commute with any involution of $M$.
\end{proof}

\begin{lem}
\label{lem:babyconj}
Suppose that $G = \mathbb{B}$, and let $R$ and $M$ be maximal subgroups of $G$ isomorphic to $\mathrm{Fi}_{23}$. Additionally, let $s \in R$ with $|s| = 23$. Then there exists $f \in R \cap M$ such that $d(s,f) \le 1$.
\end{lem}

\begin{proof}
Let $S:=\langle s \rangle$. If $S \le M$, then we can set $f = s$. Otherwise, $S \cap M = 1$, and since $|R|\,|M| > |G|$, there exists $f \in (M \cap R) \setminus \{1\}$. As $f \notin S = C_R(s)$, we observe that $s \sim f$.
\end{proof}

\begin{lem}
\label{lem:babyneighb}
Suppose that $G = \mathbb{B}$, and let $s,y \in G \setminus \{1\}$, with $|s| = 23$. Then $d(s,y) \le 3$.
\end{lem}

\begin{proof}
The element $s$ generates a Sylow $23$-subgroup $S$ of $G$ and lies in maximal subgroups $R \cong \mathrm{Fi}_{23}$ and $T \cong \nonsplit{2^{1+22}}{\mathrm{Co}_2}$. Each maximal subgroup of $G$ that contains $s$ and is not $G$-conjugate to $R$ or $T$ has shape $47 \sd 23$. Moreover, $C_R(s) = S$ and (using Proposition~\ref{prop:spormaxz} and the \textsc{Atlas}) $C_G(s) = C_T(s) = Z(T) \times S \cong C_2 \times S$. In addition, Proposition~\ref{prop:isolvertnc} shows that $y \in L \setminus Z(L)$ for some maximal subgroup $L$ of $G$. We will show that $d(s,y) \le 3$ whenever $L$ satisfies certain properties, and that we may often choose $L$ to satisfy those properties.

Suppose first that $Z(L) > 1$ and $|L|\,|T| > |G|$, so that $L \cap T \ne 1$. We may assume that $L \ne T$, else $d(s,y) \le 2$ by Lemma~\ref{lem:propernc}. If there exists a non-identity $u \in (L \cap T) \cap (Z(L) \cup Z(T))$, then $Z(L) \cup Z(T) \subseteq L \cap T$ (if $u \in Z(L)$, say, then $u \notin Z(G)$ implies that $Z(T) \le L \cap T$, and then $Z(T) \setminus Z(G) \ne \varnothing$ yields $Z(L) \le L \cap T$). Note also that $Z(L) \cup Z(T)$ is a proper subset of $\langle Z(L), Z(T) \rangle$, since $Z(L) \cap Z(T) = 1$. Thus whether or not such an element $u$ exists, $L \cap T$ contains an element $c \notin Z(L) \cup Z(T)$. Moreover, any such $c$ satisfies $d(c,y) \le 2$ by Lemma~\ref{lem:propernc}. In particular, if $S \le L \cap T$, then we can set $c = s$, and so $d(s,y) \le 2$. Assume therefore that $S \not\le L \cap T$. We claim that $s \sim c$. Indeed, if $s \nsim c$, then $c$ centralises $s$, and hence lies in $(Z(T) \times S) \setminus (Z(T) \cup S)$. Thus $c = z s^i$ for some $i \in \{1,2,\ldots,22\}$, where $z$ is the unique involution of $Z(T)$. This implies that $\langle c^2 \rangle = \langle s^{2i} \rangle = S$, and so $S \le L \cap T$, a contradiction. Therefore, $s \sim c$ and $d(s,y) \le 3$.

Next, suppose that $Z(L) = 1$, $|L|\,|R| > |G|$, and $L \not\cong R$. Then $L \cap R \ne 1$, while $S \cap L = 1$ (since no $L$ satisfying the given assumptions has order divisible by $|S|$). As $C_R(s) = S$, each $b \in L \cap R$ is adjacent to $s$ in $\nc(G)$. Moreover, $d(b,y) \le 2$ by Lemma~\ref{lem:propernc}, and so again $d(s,y) \le 3$.

Assume now that $|y| \notin \{25,47,55\}$. Using GAP, we see that there exists a maximal subgroup $K$ of $G$ containing $y$, such that either $Z(K) > 1$ and $|K|\,|T| > |G|$; or $Z(K) = 1$, $|K|\,|R| > |G|$, and $K \not\cong R$. Moreover, Proposition~\ref{prop:spormaxz} and Lemma~\ref{lem:simplenoncent}\ref{simplenoncent1} show that $K \setminus Z(K)$ contains a $G$-conjugate of $y$. Hence we can choose $L$ to be a $G$-conjugate of $K$, and $d(s,y) \le 3$ by the previous two paragraphs. To complete the proof, we will consider the remaining elements $y$.

\medskip

\noindent \textbf{Case (a)}: $|y| = 47$. Here, $C_G(y) = \langle y \rangle$, and each maximal subgroup of $G$ containing $y$ has shape $47 \sd 23$. As $S$ is a Sylow subgroup of $G$, it follows that $y \sim s'$ for some $s' \ne 1$ in a $G$-conjugate $S'$ of $S$. To complete the proof in this case, we will show that $d(s,s') \le 2$.

The element $s'$ lies in a $G$-conjugate $M$ of $R$. If $\langle s \rangle = \langle s' \rangle$, then it is clear that $d(s,s') = 2$. Otherwise, $[s,s'] \ne 1$, and so if $s$ or $s'$ lies in $R \cap M$, then $d(s,s') = 1$. Suppose finally that $s, s' \notin R \cap M$. We deduce from Lemma~\ref{lem:babyconj} that $d(s,f) = 1 = d(s',f')$ for some $f,f' \in R \cap M$. Hence there exists $t \in R \cap M$ centralising neither $s$ nor $s'$, and $(s,t,s')$ is a path in $\nc(G)$.

\medskip

\noindent \textbf{Case (b)}: $|y| \in \{25,55\}$. There is a unique $G$-conjugacy class of elements of order $|y|$, and so we may choose $L \cong \mathrm{HN} \sd 2$ if $|y| = 25$, or $L \cong (5 \sd 4) \times (\mathrm{HS} \sd 2)$ if $|y| = 55$. In either case, no involution of $G$ has a centraliser of order $|L|$, and so Proposition~\ref{prop:spormaxz} yields $Z(L) = 1$.

First suppose that there exists $g \in (L \cap R) \setminus \{1\}$. Then $g \notin {S \cup Z(L) \cup Z(R)}$, as $|S|$ does not divide $|L|$ and $Z(L) = Z(R) = 1$. Hence Lemma~\ref{lem:propernc} yields $d(s,g), d(g,y) \le 2$. Additionally, $C_G(y) = \langle y \rangle$, and so $C_G(y) \cap C_G(s) = 1$. Thus either $s \sim g$ or $d(g,y) \le 1$, and $d(s,y) \le 3$.

Suppose finally that $L \cap R = 1$. Each of $G$ and $R$ has a unique conjugacy class of elements of order $35$, and $L$ also contains elements of order $35$. Hence there exists a $G$-conjugate $M$ of $R$ such that $L \cap M$ contains an element $m$ of order $35$, and $m \sim y$ since $C_G(y) = \langle y \rangle$. Additionally, Lemma~\ref{lem:babyconj} shows that $d(s,f) \le 1$ for some $f \in R \cap M$. Note that $f \in M \setminus \langle m \rangle = M \setminus C_M(m)$, since $m \in L$ and $L \cap R = 1$. Thus $f \sim m$ and $d(s,y) \le 3$.
\end{proof}

Recall that the intersection graph $\Delta_G$ of $G$ has vertices the proper nontrivial subgroups of $G$, and its edges are the pairs of subgroups that intersect nontrivially. The result \cite[Theorem 1.1]{intgraph} on the diameter of $\Delta_G$ (for all non-abelian finite simple groups $G$) does not take into account the new information about $\mathbb{M}$ mentioned in Remark~\ref{rem:monsteroddmax}. We now correct that theorem in the case $G = \mathbb{M}$; this will be useful in the proofs of Theorems~\ref{thm:spordiamnc} and \ref{thm:nongensimple}.

\begin{prop}
\label{prop:intmonst}
Suppose that $G = \mathbb{M}$. Then the intersection graph $\Delta_G$ of $G$ is connected with diameter $5$ or $6$.
\end{prop}

\begin{proof}
Let $M_1$ and $M_2$ be maximal subgroups of $G$, with $|M_1|$ odd. By elementary arguments from the proof of \cite[Theorem 1.1]{intgraph}, it suffices to bound $d(S_1,S_2)$, where $S_1$ and $S_2$ are proper nontrivial subgroups of $M_1$ and $M_2$, respectively, such that $S_1$ lies in no maximal subgroup of even order. By Remark~\ref{rem:monsteroddmax}, $M_1$ has shape $59 \sd 29$. Let $S$ and $H$ be subgroups of $M_1$ of order $59$ and $29$, respectively, so that $H$ is a Sylow subgroup of $G$.

We observe from \cite{ATLAS} that $|N_G(H)| = 2436$. Arguing almost exactly as in the $\mathbb{B}$ case of the proof of \cite[Theorem 1.1]{intgraph}, with the aid of computations involving the GAP Character Table Library, we determine that $M_1$ is the unique maximal subgroup of $G$ containing $S$; that the maximal subgroups of $G$ containing $H$ are precisely $84$ conjugates of $M_1$, three conjugates of $\mathrm{PGL}(2,29)$, and one conjugate of $3.\mathrm{Fi}_{24}$; and that $d(S,S^g) \ge 5$ for some $g \in G$. In particular, we may assume that $S_1 = S$.

Now, let $R$ be a maximal subgroup of even order containing $H$. If $S_2$ is not conjugate to $S$, then $M_2$ can be chosen to have even order. Observe that the subgroup $D$ generated by an involution of $R$ and an involution of $M_2$ is dihedral. Hence $S_1 \sim M_1 \sim R \sim D \sim M_2 \sim S_2$, and we obtain $d(S_1,S_2) \le 5$. If instead $S_2 = S^x$ for some $x \in G$, then we similarly deduce that $S_1 \sim M_1 \sim R \sim \hat D \sim R^x \sim M_1^x \sim S_2$ for some dihedral group $\hat D$, and thus $d(S_1,S_2) \le 6$. Therefore, $\diam(\Delta_G) \in \{5,6\}$.
\end{proof}

\begin{rem}
\label{rem:intmax6}
Let $G$, $S$ and $M_1 \cong 59 \sd 29$ be as in the proof of Proposition~\ref{prop:intmonst}, and let $F$ be a maximal subgroup of $G$ of shape $3.\mathrm{Fi}_{24}$. It follows from the proof of Proposition~\ref{prop:intmonst} that $\diam(\Delta_G) = 6$ if and only if there exists $x \in G$ such that $d(S,S^x) = 6$. Since $M_1$ is the unique neighbour of $S_1$ in $\Delta_G$, this occurs if and only if $d(M_1,M_1^x) = 4$.

One possible obstruction to this occurring is the existence of elements $w, y \in G$ such that $M_1 \sim F^w \sim F^y \sim M_1^x$. Since $H$ lies in a unique conjugate of $F$, and $M_1$ contains $59$ conjugates of $H$, exactly $59$ conjugates of $F$ are adjacent to $M_1$ in $\Delta_G$. An elementary counting argument now shows that exactly $\alpha:=59|F|/|M_1| \approx 2.6 \times 10^{23}$ conjugates of $M_1$ are adjacent to $F$. Additionally, by the proof of \cite[Theorem 1.3]{liebeckshalev}, the number $\beta$ of conjugates of $F$ adjacent to $F$ is at most $\gamma:=|G:F|\sum_{r \in R}\frac{|r^G \cap F|^2}{|r^G|}$, where $R$ is a set of representatives for the $G$-conjugacy classes of elements of prime order. We calculate using the GAP Character Table Library that $\gamma \approx 4.2 \times 10^{28}$. We therefore deduce that there are at most $59\alpha\gamma \approx 6.5 \times 10^{53}$ conjugates $M_1^x$ of $M_1$ such that $M_1 \sim F^w \sim F^y \sim M_1^x$ for some $w, y \in G$. This is unfortunately not a useful upper bound, as it is larger than the number $|G:M_1| \approx 4.7 \times 10^{50}$ of conjugates of $M_1$.

Determining a better estimate for $\beta$ is difficult due to the large index of $F$ in $G$. However, if a sufficiently small upper bound for $\beta$ can be obtained, then the above argument (along with similar arguments for the other types of path between $M_1$ and $M_1^x$ of length at most three) would yield the existence of an element $x$ such that $d(M_1,M_1^x) = 4$, so that $\diam(\Delta_G) = 6$. Otherwise, deducing the exact value of $\diam(\Delta_G)$ would require more detailed information about the neighbourhoods of conjugates of $M_1$, $F$ and the maximal subgroup $\mathrm{PGL}(2,29)$ of $G$. The large order of $G$ is again a significant obstacle to obtaining this information.
\end{rem}

\begin{proof}[Proof of Theorem~\ref{thm:spordiamnc}]
For each $G \in \{\mathrm{M}_{11},\mathrm{M}_{12},\mathrm{M}_{22},\mathrm{M}_{23},\mathrm{J}_1,\mathrm{J}_2\}$, we obtain $\diam(\nc(G))$ using Magma. In all other cases, for $i \in \{1,2\}$, let $x_i$ be an element of $G \setminus \{1\}$ that lies in a maximal subgroup of even order, say $L_i$. By Proposition~\ref{prop:sporinvol}, $d(x_i,a_i) \le 1$ for some involution $a_i$ of $G$, and so (using Proposition~\ref{prop:isolvertnc} if $|x_i| = 2$) we may choose $L_i$ so that $x_i \notin Z(L_i)$. Lemma~\ref{lem:ncmaxsimple} now gives $d(x_1,x_2) \le 4$. Hence $\diam(\nc(G)) \le 4$ if each maximal subgroup of $G$ has even order.

It remains to consider the case where $G$ has a maximal subgroup $K$ of odd order. By Theorem~\ref{thm:simpleoddmax} and Remark~\ref{rem:monsteroddmax}, $G \in \{\mathrm{Th},\mathbb{B},\mathbb{M}\}$. Let $g,h \in G \setminus \{1\}$ such that $h$ lies in no maximal subgroup of even order.

First assume that $G = \mathrm{Th}$. Then $K \cong 31 \sd 15$. Let $m \in K$. If $|m| = 31$, then $m$ lies in a maximal subgroup of $G$ of shape $\nonsplit{2^5}{\mathrm{PSL}(5,2)}$, and otherwise $|C_G(m)|$ is even. Thus each element of $G$ lies in a maximal subgroup of even order, and so the element $h$ does not exist. Hence the first paragraph of the proof yields $\diam(\nc(G)) \le 4$.

Suppose next that $G = \mathbb{B}$. In this case, $K \cong 47 \sd 23$, $|h| = 47$, and $C_G(h) = \langle h \rangle$. Thus $h \sim s$ for each $s \in K$ of order $23$. As $d(s,g) \le 3$ by Lemma~\ref{lem:babyneighb}, we conclude that $d(h,g) \le 4$. It follows that $\diam(\nc(G)) \le 4$. In addition, the intersection graph of $G$ has diameter $5$ \cite[Theorem 1.1]{intgraph}, and so Lemma~\ref{lem:intncgraphs} yields $\diam(\nc(G)) \ge 4$. Therefore, $\diam(\nc(G)) = 4$.

Finally, assume that $G = \mathbb{M}$, so that $K \cong 59 \sd 29$, $|h| = 59$, and $C_G(h) = \langle h \rangle$. Thus $h \sim s$ for each $s \in K$ of order $29$. Let $R$ be a maximal subgroup of $G$ of even order containing $s$, and note that $\langle s \rangle$ is a Sylow subgroup of $G$. If $g$ is not conjugate to $h$, then, as in the first paragraph of the proof, $d(g,a) \le 1$ for some involution $a$ of $G$, and there exists a maximal subgroup $L$ of $G$ such that $g,a \in L \setminus Z(L)$. By Lemma~\ref{lem:simplenoncent}\ref{simplenoncent3} and the fact that two involutions generate a dihedral group, $R$ contains an involution $b$ such that $a \sim b$. Moreover, $|C_G(s)| = 87$, which implies that $s \sim b$. Hence $h \sim s \sim b \sim a \sim g$, and so $d(h,g) \le 4$. If instead $g = h^j$ for some $j \in G$, then we similarly deduce the existence of involutions $a'$ and $b'$ such that $h \sim s \sim b' \sim a' \sim s^j \sim g$, and so $d(h,g) \le 5$. Therefore, $\diam(\nc(G)) \le 5$. As the intersection graph of $G$ has diameter at least $5$ by Proposition~\ref{prop:intmonst}, Lemma~\ref{lem:intncgraphs} implies that $\diam(\nc(G)) \in \{4,5\}$.
\end{proof}

Determining the exact diameter of $\nc(\mathbb{M})$ (and similarly for the non-generating graph of $\mathbb{M}$ in Theorem~\ref{thm:nongensimple}) is difficult for reasons similar to those discussed in Remark~\ref{rem:intmax6}.

\subsection{Exceptional groups of Lie type}
\label{subsec:exceptionalnc}

We now prove Theorem~\ref{thm:ncsimple} in the case where $G$ is an exceptional group of Lie type. In view of Proposition~\ref{prop:ncsimpleeven} and Theorem~\ref{thm:simpleoddmax}, it remains to prove the following. Here, $q$ denotes a prime power, and we do not include ${}^2G_2(3)' \cong \mathrm{PSL}(2,8)$; for that group, see Theorem~\ref{thm:linearnc}.

\begin{thm}
\label{thm:ncexcep}
Let $G$ be a finite simple group. If $G \cong \mathrm{Sz}(q)$, or if $q$ is odd and $G \in \{G_2(q), {}^2G_2(q), {}^3\!D_4(q), F_4(q), E_8(q)\}$, then $\nc(G)$ is connected with diameter at most $4$.
\end{thm}

Given a group $X$, let $Q_X$ be the (characteristic) subgroup generated by the involutions of $X$.

\begin{lem}
\label{lem:quasicirc}
Let $S:=\mathrm{SL}(2,q)$, and let $H$ be a quasisimple group containing a subgroup $K$, such that $K \cong S$ and $Z(H) = Z(K)$. Additionally, let $R$ be the central product $S \circ H$, and let $\pi$ be the natural epimorphism $S \times H \to R$. If $Q_R \ne R$, then $q = 3$, $Q_R$ is a maximal subgroup of $R$ that contains $(H)\pi$, and $C_R(Q_R) = Z(R)$.
\end{lem}

\begin{proof}
We observe that $\tilde S := (S)\pi$ and $\tilde K:=(K)\pi$ are normal subgroups of $J:=S \circ K = {(S \times K)\pi}$, and $\tilde S$ and $\tilde H:=(H)\pi$ are normal subgroups of $R$, with $\tilde S \cong S$, $\tilde K \cong K$ and $\tilde H \cong H$. Additionally, $[\tilde H, \tilde S] = ([H,S])\pi = 1$ and $\tilde H \tilde S = R$. Furthermore, the centres of $R$, $J$, $\tilde H$, $\tilde S$ and $\tilde K$ all coincide, and this common centre $\mathcal{Z}$ of order $(2,q-1)$ is equal to $\tilde H \cap \tilde S$. We divide the rest of the proof into two cases.

\medskip

\noindent \textbf{Case (a)}: $q \ne 3$. We first show that $Q_J = J$. This is clear if $q = 2$, so assume otherwise. Considering $K$ as a copy of $\mathrm{SL}(2,q)$, let $L:=\{(A,B) \in S \times K \mid A^2 = B^2 = -I_2\}$, where $I_2$ is the $2 \times 2$ identity matrix. The choices for $A$ and $B$ include the matrix $\begin{pmatrix}
0&a\\
-a^{-1}&0\\
\end{pmatrix}$ for each $a \in \mathbb{F}_q^\times$. As $S$ is quasisimple, and a quasisimple group's proper normal subgroups are precisely its central subgroups, we see that $\langle L \rangle^{S \times K}$ projects onto each of $S$ and $K$, and that the kernel of each projection map is isomorphic to $S \cong K$. Thus $\langle L \rangle^{S \times K} = S \times K$. Since the image under $\pi$ of each element of $L$ is an involution, it follows that $Q_J = J$.

Now, since $\tilde H$ is quasisimple and $\mathcal{Z} < \tilde K \le \tilde H$, no proper normal subgroup of $\tilde H$ contains $\tilde K$. Hence no proper normal subgroup of $R$ contains $J$. As $J = Q_J \le Q_R \trianglelefteq R$, we obtain $Q_R = R$.

\medskip

\noindent \textbf{Case (b)}: $q = 3$. Here, $Q_J$ is the image under $\pi$ of the direct product of the normal Sylow $2$-subgroups of $S$ and $K$, which are isomorphic to the quaternion group $Q_8$. Thus $\tilde S \cap Q_J$ is maximal in $\tilde S$, and $\tilde K \cap Q_J > \mathcal{Z}$. As the normal subgroup $\tilde H \cap Q_R$ of $\tilde H$ contains $\tilde K \cap Q_J$, it follows that $\tilde H \le Q_R$. Since $R = \tilde H \tilde S$, we deduce that $Q_R = \tilde H \tilde S \cap Q_R = \tilde H(\tilde S \cap Q_R)$.

If $\tilde S \cap Q_R = \tilde S$, then $Q_R = R$. We will therefore assume that $\tilde S \cap Q_R < \tilde S$. Then $\tilde S \cap Q_R$ is the maximal subgroup $\tilde S \cap Q_J$ of $\tilde S$, and $Q_R = \tilde H(\tilde S \cap Q_J)$, which is maximal in $\tilde H \tilde S = R$. Since $[\tilde H, \tilde S] = 1$, we obtain $\mathcal{Z} \le C_R(Q_R) = C_{\tilde H \tilde S}(\tilde H(\tilde S \cap Q_J)) \le Z(\tilde H)C_{\tilde S}(\tilde S \cap Q_J).$ The centraliser $C_{\tilde S}(\tilde S \cap Q_J)$ is equal to $\mathcal{Z} = Z(\tilde H) = Z(R)$, and hence so is $C_R(Q_R)$.
\end{proof}

For an example where $Q_R \ne R$, take $H$ to be the quasisimple group $\mathrm{SL}(2,9)$.

\begin{prop}
\label{prop:quasiinvol}
Let $H$ be a finite quasisimple group. If $|Z(H)| = 2$, then assume that $H/Z(H)$ is not isomorphic to $A_7$, or to $\mathrm{PSL}(2,q)$ for any odd $q$. Then $Q_H = H$.
\end{prop}

\begin{proof}
This is a straightforward consequence of the main theorem of \cite{griess}, which implies that there exists an involution $a \in H \setminus Z(H)$ if $|Z(H)|$ is even. Clearly, $a$ also exists if $|Z(H)|$ is odd. Each proper normal subgroup of the quasisimple group $H$ lies in $Z(H)$, and so $Q_H = H$.
\end{proof}

\begin{lem}
\label{lem:excepinvolodd}
Let $G$ be a simple group in $\{G_2(q), {}^2G_2(q), {}^3\!D_4(q), F_4(q), E_8(q)\}$ with $q$ odd \textup{(}so that $q \ge 27$ if $G = {}^2G_2(q)$\textup{)}, let $a$ be an involution of $G$, and let $Y:=C_G(a)$. Then $C_Y(Q_Y) = Z(Y) = \langle a \rangle$.
\end{lem}

\begin{proof}
We observe from \cite[pp.~171--177]{gls3} that $Z(Y) = \langle a \rangle$, and either:
\begin{enumerate}[label={(\alph*)},font=\upshape]
\item $G \cong {}^2G_2(q)$ and $Y \cong C_2 \times \mathrm{PSL}(2,q)$;
\item $G \cong F_4(q)$ and $Y \cong \mathrm{Spin}(9,q)$;
\item $G \cong E_8(q)$ and $Y \cong J \sd 2$, where $J := (\mathrm{Spin}^+(16,q)/A)$, and $A$ is a central subgroup of $\mathrm{Spin}^+(16,q)$ of order $2$; or
\item $Y = R \sd A$, with $|A| = 2$ and $R \cong \mathrm{SL}(2,q) \circ H$, where $(G,H)$ lies in the set \[\{(G_2(q),\mathrm{SL}(2,q)),({}^3\!D_4(q),\mathrm{SL}(2,q^3)),(F_4(q),\mathrm{Sp}(6,q)),(E_8(q),E_7(q)_{\mathrm{sc}})\}.\] Additionally, $Z(Y) = Z(R)$, and $A$ induces an outer automorphism on the image of $H$ under the natural epimorphism $\pi: \mathrm{SL}(2,q) \times H \to R$.
\end{enumerate}
We will divide the remainder of the proof into the cases above. In each case, since $Z(Y) = \langle a \rangle$, it remains to show that $C_Y(Q_Y) = Z(Y)$. Note that this is certainly the case if $Q_Y = Y$.

\medskip

\noindent \textbf{Cases (a)--(c)}. By applying Proposition \ref{prop:quasiinvol} to $Y$ in Case (b) and $J$ in Case (c), we easily deduce that $Q_Y = Y$ in each case.

\medskip

\noindent \textbf{Case (d)}. If $G \cong G_2(3)$, then $Y \cong \mathrm{SO}^+(4,3)$ \cite[p.~125]{wilson}, and $Q_Y = Y$. Suppose therefore that $G \not\cong G_2(3)$, and that $Q_Y < Y$. Notice that $Q_Y$ contains $\langle Q_R, A \rangle = Q_R \sd A$, and thus $Q_R < R$.

In order to apply Lemma~\ref{lem:quasicirc}, we will show that the quasisimple group $H$ contains a subgroup $K$ such that $K \cong \mathrm{SL}(2,q)$ and $Z(H) = Z(K)$. Such $K$ clearly exists if there exists $L \le H$ with $Z(H) = Z(L)$ and $L \cong \mathrm{SL}(2,q^i)$ for some positive integer $i$. Hence this is the case if $H \in \{\mathrm{SL}(2,q),\mathrm{SL}(2,q^3)\}$. If instead $H = E_7(q)_{\mathrm{sc}}$, then $H$ contains a subgroup $L \cong \mathrm{SL}(2,q^7)$, with $Z(H) = Z(L)$ \cite[p.~927, item (3)]{heide}. Finally, if $H = \mathrm{Sp}(6,q)$, then let $V_1$, $V_2$ and $V_3$ be non-degenerate two-dimensional subspaces of the symplectic space $V:=\mathbb{F}_q^6$, such that $V = V_1 \perp V_2 \perp V_3$.
Then the intersection of the stabilisers in $H$ of $V_1$, $V_2$ and $V_3$ is isomorphic to $\mathrm{Sp}(2,q)^3 = \mathrm{SL}(2,q)^3$. The diagonal subgroup of this direct product is isomorphic to $\mathrm{SL}(2,q)$ and has centre $Z(H)$.

Lemma~\ref{lem:quasicirc} now shows that $Q_R$ is a maximal subgroup of $R$, that $(H)\pi \le Q_R$, and that $C_R(Q_R) = Z(R)$. It follows from this first fact that the proper subgroup $Q_Y$ of $Y$ is maximal and equal to $Q_R \sd A$. Since $(H)\pi \le Q_R \le R = (\mathrm{SL}(2,q))\pi(H)\pi$ and $[(\mathrm{SL}(2,q))\pi, (H)\pi] = 1$, each inner automorphism of $Q_R$ acts on $(H)\pi$ as an inner automorphism of $(H)\pi$. Thus the outer automorphism of $(H)\pi$ induced by $A$ does not extend to an inner automorphism of $Q_R$. Therefore, $C_Y(Q_R) = C_R(Q_R) = Z(R)$, which is equal to $Z(Y)$ from the start of the proof. As $Z(Y) \le C_Y(Q_Y) \le C_Y(Q_R)$, we conclude that $C_Y(Q_Y) = Z(Y)$.
\end{proof}

\begin{proof}[Proof of Theorem~\ref{thm:ncexcep}]
Theorem \ref{thm:simpleoddmax} implies that each maximal subgroup of $G$ has even order, and so it suffices by Proposition~\ref{prop:isolvertnc} and Lemma~\ref{lem:ncmaxsimple} to show that each non-involution $x \in G \setminus \{1\}$ is adjacent in $\nc(G)$ to some involution. Let $M$ be a maximal subgroup of $G$ containing $x$. Assume first that $G \not\cong \mathrm{Sz}(q)$, so that $q$ is odd, and let $a$ be an involution of $M$. If $x \sim a$, then we are done. Otherwise, $x \in C_G(a)$. By Lemma~\ref{lem:excepinvolodd}, $a$ is the unique non-identity element of $C_G(a)$ that centralises each involution of $C_G(a)$. Hence $x \sim a'$ for some involution $a' \in C_G(a)$.

Assume now that $G \cong \mathrm{Sz}(q)$. By Theorem 4 on page 122 of \cite{suzuki}, $C_G(y)$ is nilpotent for each $y \in G \setminus \{1\}$. As the only non-abelian finite simple groups with nilpotent maximal subgroups are $\mathrm{PSL}(2,p)$ for various primes $p$ (see \cite[p.~183]{rosemax}), we deduce that $Z(M) = 1$ and $C_M(x) < M$. If $M \cong \mathrm{Sz}(q_0)$, for some proper power $q_0$ of $2$ dividing $q$, then $M = Q_M$. It follows in this case that $M$ contains an involution $a$ not centralising $x$, and so $x \sim a$. If instead $M \not\cong \mathrm{Sz}(q_0)$, then by \S4, page 133 and Theorem 9 on page 137 of \cite{suzuki}, $M$ is conjugate in $G$ to either:
\begin{enumerate}[label={(\alph*)},font=\upshape]
\item the Frobenius group $W:=S \sd C_{q-1}$, where $S$ is a Sylow $2$-subgroup of $G$;
\item the dihedral group $X:=D_{2(q-1)}$; or
\item a Frobenius group $Y_\pm:=C_{q\pm\sqrt{2q}+1} \sd C_4$.
\end{enumerate}

If $M \cong X$, then again $Q_M = M$ and $x \sim a$ for some involution $a \in M$. If instead $M \cong Y_\pm$, then let $N$ be the normal subgroup $C_{q \pm \sqrt{2q}+1}$ of $M$. As $M$ is Frobenius, there exists a complement $H$ of $N$ with $x \notin H$. No non-identity element of $H$ centralises $x$, and so $x$ is adjacent in $\nc(G)$ to the unique involution of $H$.

Suppose finally that $M \cong W$. If $x \notin S$, then $x$ centralises no non-identity element of $S$, and so $x$ is adjacent in $\nc(G)$ to each involution of $S$. Now, $Z(S)$ consists of all elements of $S$ of order at most $2$, and each element of $S \setminus Z(S)$ has order $4$, by Lemma 1 on page 112 and Theorem 7 on page 133 of \cite{suzuki}. Additionally, $G$ has a unique conjugacy class of cyclic subgroups of order $4$ (see the final sentence of page 121 and Proposition 18 on page 126 of \cite{suzuki}), and so if the non-involution $x$ lies in $S$, then it also lies in a maximal subgroup $L \cong Y_\pm$. By the previous paragraph, $x$ is adjacent in $\nc(G)$ to an involution of $L$.
\end{proof}

Magma calculations show that $\diam(\nc(\mathrm{Sz}(8))) = 3$ and $\diam(\nc(G_2(3))) = 2$. We do not know if any finite simple exceptional group $G$ satisfies $\diam(\nc(G)) > 3$. Observe also that the proof of Theorem~\ref{thm:ncexcep} relies on useful properties of the involution centralisers in $G$ (and of $\mathrm{Sz}(q)$ and its maximal subgroups). Indeed, in the $q$ odd case, we have considered precisely those finite simple groups whose involutions are all centralised by maximal subgroups of \emph{maximal rank} (see \cite[Theorem 4.5.1]{gls3} and \cite{liebeckmaxrank}). The involution centralisers for the remaining exceptional groups have more complicated structures (see for example \cite{aschbacherseitz} and \cite[Theorem 4.5.1]{gls3}), so alternative methods would be necessary to prove an analogue of Theorem~\ref{thm:ncexcep} for those groups.

\section{Matrices and subspace stabilisers}
\label{sec:matprel}

We now consider various matrices and subspace stabilisers. The results here will be useful when proving Theorem~\ref{thm:ncsimple} for linear and unitary groups. Let $n$ be a positive integer, $q$ a prime power, and $I_k$ the $k \times k$ identity matrix over $\mathbb{F}_q$, for a positive integer $k$. Additionally, let $E_{i,j}$ be the $n \times n$ matrix with $(i,j)$ entry equal to $1$ and all other entries equal to $0$.

\subsection{Matrices and subspace stabilisers in linear groups}
\label{subsec:matlin}

\begin{defn}[{\cite[p.~148 \& pp.~161--162]{perlis}}]
\label{def:compmat}
Let $f(x):=x^n-\beta_{n}x^{n-1}-\ldots-\beta_2x-\beta_1$ be a (monic) polynomial in $\mathbb{F}_q[x]$. The \emph{companion matrix}\index{companion matrix} of $f$ is the $n \times n$ matrix
\[C(f):=\left( \begin{matrix}
0&I_{n-1}\\
A&B
\end{matrix} \right),\]
where $A:=(\begin{matrix} \beta_1 \end{matrix})$, $B:=(\begin{matrix} \beta_2 & \cdots & \beta_n \end{matrix})$, and $C(f) = A$ when $n = 1$.
If $f$ is irreducible over $\mathbb{F}_q$, then for each positive integer $k$, the \emph{hypercompanion matrix}\index{hypercompanion matrix} of $f^k$ is the 
$kn \times kn$ matrix
\[C_k(f):=\left( \begin{matrix}
C(f)&E_{n,1}&0&\cdots&0\\
0&C(f)&E_{n,1}&\cdots&0\\
\vdots&\vdots&\ddots&\ddots&\vdots\\
0&0&0&\ddots&E_{n,1}\\
0&0&0&\cdots&C(f)
\end{matrix} \right).\]
\end{defn}

Since a subgroup $\langle A \rangle$ of $\mathrm{GL}(n,q)$ is irreducible if and only if the characteristic polynomial $\chi_A$ of $A$ is irreducible, the following lemma will shed light on certain irreducible subgroups.

\begin{lem}
\label{lem:binomialdiv}
Let $a \in \mathbb{F}_q^\times$. Then the set of irreducible factors in $\mathbb{F}_q[x]$ of the binomial $x^{q-1} - a$ is $\{x^{|a|}-b \mid b \in \mathbb{F}_q^\times, b^{(q-1)/|a|} = a\}$.
\end{lem}

\begin{proof}
For each positive integer $k$, let $\phi_k$ be the homomorphism $s \mapsto s^k$ of $\mathbb{F}_q^\times$, and let $H_k$ be the image of $\phi_k$, of order $(q-1)/(k,q-1)$. Additionally, let $r:=|a|$ and $t:=(q-1)/r$. Then $H_t = \langle a \rangle$. Since $|\ker(\phi_t)| = t$, there are exactly $t$ distinct roots $b_1, \ldots, b_t \in \mathbb{F}_q^\times$ of the binomial $f(x):=x^t-a$. Hence $f(x) = \prod_{i=1}^t (x^t-b_i)$, and $x^{q-1}-a = f(x^r) = \prod_{i=1}^t (x^r-b_i).$

It remains to show that the binomial $x^r-b_i$ is irreducible over $\mathbb{F}_q$ for each $i$. Since $b_i^t = a$, we see that $r=|a|$ divides $|b_i|$. Additionally, if $r$ is divisible by $4$, then so is $q-1$. It therefore suffices by \cite[Theorems 3.3 \& 3.35]{lidl} to prove that $r$ is coprime to $u:=(q-1)/|b_i|$.

Since $|H_u| = |b_i|$, there exists $c_i \in \mathbb{F}_q^\times$ with $c_i^u = b_i$, and hence $c_i^{ut} = a$. Thus $a \in H_{ut}$, and so $r$ divides $|H_{ut}| = {(q-1)/(ut,q-1)} = tr/(ut,tr) = r/(u,r)$. Hence $(u,r) = 1$.
\end{proof}

Recall that matrices $A, B \in \mathrm{GL}(n,q)$ are similar if and only if $\chi_A = \chi_B$.

\begin{lem}
\label{lem:irredmatlowdim}
Let $A \in \mathrm{SL}(n,q)$, and suppose that $A^{q^2-1} \in Z(\mathrm{SL}(n,q))$, and that $\langle A \rangle$ acts irreducibly on $\mathbb{F}_q^n$. Then $n \in \{1,2\}$. If, in addition, $n = 2$ and $A^{q-1} \in Z(\mathrm{SL}(2,q))$, then $q \equiv 3 \pmod 4$.
\end{lem}

\begin{proof}Since $A^{q^2-1} \in Z(\mathrm{SL}(n,q))$, there exists $a \in \mathbb{F}_q^\times$ such that $A^{q^2-1} - aI_n = 0$. Thus $\chi_A$, which is irreducible over $\mathbb{F}_q$, divides $x^{q^2-1} - a \in \mathbb{F}_q[x]$. Furthermore, $\mathbb{F}_{q^n}$ is the splitting field for $\chi_A$ over $\mathbb{F}_q$, and $A$ is similar to the companion matrix $C(\chi_A)$. To prove the required result, we will study $\chi_A$ as a polynomial in $\mathbb{F}_{q^2}[x]$. By Lemma~\ref{lem:binomialdiv}, the set of irreducible factors in $\mathbb{F}_{q^2}[x]$ of $x^{q^2-1} - a$ is $\{x^{|a|}-b \mid b \in \mathbb{F}_{q^2}^\times, b^{(q^2-1)/|a|} = a\}$.

If $n$ is odd, then $\chi_A$ is irreducible over $\mathbb{F}_{q^2}$. Thus $n = |a|$, and $\chi_A = x^n-b$ for some $b \in \mathbb{F}_q^\times$ with $b^{(q^2-1)/n} = a$. Since $\det(C(x^n-b)) = b$, we deduce that $b = 1$. Hence $a = 1$ and $n = 1$.

If instead $n$ is even, then each irreducible factor of $\chi_A$ in $\mathbb{F}_{q^2}$ has degree $n/2$. In particular, $|a| = n/2$, and there exist $b_1,b_2 \in \mathbb{F}_{q^2}^\times$ such that \[\chi_A(x) = (x^{n/2}-b_1)(x^{n/2}-b_2)= {x^n - (b_1+b_2)x^{n/2} + b_1b_2},\] with $b_1^{(q^2-1)/(n/2)} = b_2^{(q^2-1)/(n/2)} = a$ and $b_1b_2, b_1+b_2 \in \mathbb{F}_q$. Here, $\det(C(\chi_A)) = b_1b_2$, and thus $b_2 = b_1^{-1}$. Therefore, $a = a^{-1}$, and hence $n/2 = |a| \in \{1,2\}$. In particular, if $q$ is even, then $a \in \mathbb{F}_q^\times$ has odd order, and $n = 2$. Otherwise, arguing as in \cite[pp.~569--570]{carlitz} shows that the irreducible polynomial $x^n-(b_1+b_2)x^{n/2} + 1 \in \mathbb{F}_q[x]$ cannot have degree $4$, and again $n = 2$.

Finally, if $n = 2$ and $A^{q-1} \in Z(\mathrm{SL}(2,q))$, then $A^{q-1} - cI_2 = 0$ for some $c \in \{-1,1\}$. Hence the irreducible polynomial $\chi_A$ divides $x^{q-1} - c$. It follows from Lemma~\ref{lem:binomialdiv} and the previous paragraph that $\chi_A(x) = x^2+1$, which is irreducible if and only if $q \equiv 3 \pmod 4$.
\end{proof}

Let $H:=\mathrm{SL}(n,q)$, with $n \ge 2$, and $q \ge 4$ if $n = 2$, so that $H/Z(H)$ is simple. Additionally, let $V:=\mathbb{F}_q^n$ be the natural module for $H$, and $H_X$ the stabiliser in $H$ of a subspace $X$ of $V$.

\begin{lem}
\label{lem:inter1stab}
Let $X$ and $Y$ be distinct one-dimensional subspaces of $V$. Then:
\begin{enumerate}[label={(\roman*)},font=\upshape]
\item \label{inter1stab:p1} no commutator in $H_X$ is a non-identity scalar matrix; 
\item $C_H(H_X) = Z(H)$; \label{inter1stab:p4} and
\item \label{inter1stab:p3} $C_H(H_X \cap H_Y) = \begin{cases} H_X \cap H_Y, & \text { if } n = 2, \text { or } n = 3 \text { and } q = 2,\\ Z(H), & \text{ otherwise.}\end{cases}$
\end{enumerate}
\end{lem}

\begin{proof} Up to conjugacy by a fixed element of $H$, the stabiliser $H_X$ consists of the matrices in $H$ with a first row of the form $\begin{matrix}(\lambda_1 & 0 & 0 & \cdots & 0)\end{matrix}$, with $\lambda_1 \in \mathbb{F}_q^\times$, and $H_Y$ is the set of matrices with a second row of the form $\begin{matrix}(0 & \lambda_2 & 0 & \cdots & 0)\end{matrix}$, with $\lambda_2 \in \mathbb{F}_q^\times$. The restriction of a matrix in $H_X$ to its $(1,1)$ entry is a homomorphism from $H_X$ to the abelian group $\mathbb{F}_q^\times$, and its kernel clearly contains all commutators in $H_X$. Thus (i) holds. We split the proof of (ii)--(iii) into two cases. In each case, let $W \in C_H(H_X \cap H_Y)$.

\medskip

\noindent \textbf{Case (a)}: $n = 2$. Let $\omega$ be a primitive element of $\mathbb{F}_q$, $A:=\mathrm{diag}(\omega,\omega^{-1}) \in H_X \cap H_Y$, and $D:=I_2+E_{2,1} \in H_X$. The equality $AW = WA$ gives $W \in H_X \cap H_Y$, which is abelian. Hence we obtain (iii). If $W \in C_H(X)$, then $DW = WD$, and we deduce that $W \in Z(H)$, hence (ii).

\medskip

\noindent \textbf{Case (b)}: $n \ge 3$. It is easy to check the result if $n = 3$ and $q = 2$, so assume otherwise. If $n = 3$, then let $b \in \mathbb{F}_q^\times \setminus \{1\}$ and $B:=\mathrm{diag}(1,b,b^{-1}), B':=\mathrm{diag}(b,1,b^{-1}) \in H_X \cap H_Y$. We deduce from the equalities $BW = WB$ and $B'W = WB'$ that $W$ is diagonal. If instead $n \ge 4$, then let $r$ and $s$ be distinct integers with $3 \le r \le n$ and $1 \le s \le n$. Then $C_{r,s}:=I_n+E_{r,s} \le H_X \cap H_Y$, and the equalities $C_{r,s}W = WC_{r,s}$ (for all $r$ and $s$) imply that $W$ is again diagonal.

In each case, let $F \in H_X \cap H_Y$ be the matrix obtained from $I_n$ by setting all entries in the final row to $1$. As $W$ is diagonal, the equality $WF = FW$ yields $W \in Z(H)$. Thus we have proved (iii), and (ii) follows immediately.
\end{proof}

\subsection{Subspace stabilisers in unitary groups}
\label{subsec:unitaryspaces}

Throughout this subsection, let $n \ge 3$, with $q \ge 3$ if $n = 3$, so that $G:=\mathrm{PSU}(n,q)$ is simple. Additionally, let $V$ be the natural module $\mathbb{F}_{q^2}^n$ for $\mathrm{SU}(n,q)$, equipped with the unitary form with Gram matrix $I_n$. Then a matrix $A \in \mathrm{SL}(n,q^2)$ lies in $\mathrm{SU}(n,q)$ if and only if $A A^{\sigma T} = I_n$, where $\sigma$ is the automorphism $\alpha \mapsto \alpha^q$ of $\mathbb{F}_{q^2}$. The following theorem will be useful when considering $\nc(G)$. Here, $G_{(\Delta)}$ denotes the pointwise stabiliser in $G$ of a set $\Delta$ of subspaces of $V$.

\begin{thm}
\label{thm:unisubspaces}
Let $\Delta$ be a set of $k$-dimensional subspaces of $V$.
\begin{enumerate}[label={(\roman*)},font=\upshape]
\item \label{unisubspaces1} Suppose that $|\Delta| < \lceil n/k \rceil$. If $(q+1) \nmid n$, or if $k \nmid (n-1)$, then $G_{(\Delta)} \ne 1$.
\item Suppose that $|\Delta| < \lceil n/k \rceil-1$. Then $G_{(\Delta)} \ne 1$.
\end{enumerate}
\end{thm}

An analogue of the following lemma holds for certain classical algebraic groups; see \cite[\S4.1]{bgs}.

\begin{lem}
\label{lem:unicodimone}
Let $W$ be an $(n-1)$-dimensional subspace of $V$. Then $\mathrm{SU}(n,q)$ contains a non-scalar matrix that acts as a scalar on $W$ if and only if $W^\perp$ is totally singular or $(q+1) \nmid n$.
\end{lem}

\begin{proof}
First, $\mathrm{SU}(n,q)$ has two orbits on the set of one-dimensional subspaces of $V$: the non-degenerate subspaces, and the totally singular subspaces (cf.~Witt's Lemma). Thus it suffices to consider two choices for $W$, one for each possible orbit containing $W^\perp$. Let $\{e_1,\ldots,e_n\}$ be the basis for $V$ corresponding to the Gram matrix $I_n$, and let $\omega$ be a primitive element of $\mathbb{F}_{q^2}$.

\medskip

\noindent \textbf{Case (a)}: $W^\perp$ is non-degenerate. We may assume that $W^\perp = \langle e_1 \rangle$, so that $W = \langle e_2, e_3, \ldots, e_n \rangle$. Any matrix in $\mathrm{SL}(n,q^2)$ that acts as a scalar on $W$ also stabilises $W^\perp$, and hence is equal to $B_\alpha:= \mathrm{diag}(\alpha^{-(n-1)},\alpha,\ldots,\alpha)$ for some $\alpha \in \mathbb{F}_{q^2}^\times$. Moreover, $B_\alpha \in \mathrm{SU}(n,q)$ if and only if $\alpha^{q+1} = 1$.

Let $\lambda:=\omega^{q-1}$. Observe that if $(q+1) \nmid n$, then $B_\lambda$ is a non-scalar matrix in $\mathrm{SU}(n,q)$ that acts as $\lambda$ on $W$. If instead $q+1 \mid n$, then any scalar $\alpha \in \mathbb{F}_{q^2}^\times$ satisfying $\alpha^{q+1} = 1$ also satisfies $\alpha^{-(n-1)} = \alpha$. Therefore, in this case, no non-scalar matrix in $\mathrm{SU}(n,q)$ acts as a scalar on $W$.

\medskip

\noindent \textbf{Case (b)}: $W^\perp$ is totally singular. If $q$ is even, then let $\gamma:= \omega^{q+1}$ and $\delta:=1$, and otherwise, let $\gamma:=\omega$ and $\delta:=\omega^{(q-1)/2}$. Then $\delta^{q+1} = -1$ for all $q$, and so $\delta e_1+e_2$ is singular. Hence we may assume that $W^\perp = \langle \delta e_1+e_2 \rangle$, so that $W = \langle \delta e_1+e_2, e_3,\ldots,e_n \rangle$. The direct sum of $\left( \begin{matrix}
1+\gamma\delta&\gamma\\
-\gamma\delta^2&1-\gamma\delta\\
\end{matrix} \right)$ and $I_{n-2}$ is a non-scalar matrix in $\mathrm{SU}(n,q)$ that acts trivially on $W$.
\end{proof}

\begin{cor}
\label{cor:unicodimtwo}
Let $U$ be an $(n-2)$-dimensional subspace of $V$. Then $\mathrm{SU}(n,q)$ contains a non-scalar matrix that acts as a scalar on $U$.
\end{cor}

\begin{proof}
Observe that $U^\perp$ contains a one-dimensional totally singular subspace $X$. By Lemma~\ref{lem:unicodimone}, some non-scalar matrix in $\mathrm{SU}(n,q)$ acts as a scalar on $X^\perp$, and hence on $U \subseteq X^\perp$.
\end{proof}

\begin{proof}[Proof of Theorem~\ref{thm:unisubspaces}]
If $|\Delta| < \lceil n/k \rceil$, then the subspace $\langle \Delta \rangle$ of $V$ lies in an $(n-1)$-dimensional subspace $W$. Lemma~\ref{lem:unicodimone} shows that if $(q+1) \nmid n$, then $\mathrm{SU}(n,q)$ contains a non-scalar matrix that acts as a scalar on $W$, and hence on $\langle \Delta \rangle$. Therefore, $G_{(\Delta)} \ne 1$.

Finally, assume that $|\Delta| < \lceil n/k \rceil-1$, or that $|\Delta| = \lceil n/k \rceil - 1$ and $k \nmid (n-1)$. Then $\langle \Delta \rangle$ lies in an ${(n-2)}$-dimensional subspace $U$ of $V$. By Corollary~\ref{cor:unicodimtwo}, $\mathrm{SU}(n,q)$ contains a non-scalar matrix that acts as a scalar on $U$, and hence on $\langle \Delta \rangle$. We again conclude that ${G_{(\Delta)} \ne 1}$.
\end{proof}

\section{Linear and unitary groups and the non-generating graph}
\label{sec:nclinearuni}

In this section, we complete the proof of Theorem~\ref{thm:ncsimple} by showing that it holds for linear and unitary groups. In addition, we prove Theorem~\ref{thm:nongensimple}. Define $n$, $q$, $I_k$ and $E_{i,j}$ as in \S\ref{sec:matprel}, and let $H \in \{\mathrm{SL}(n,q),\mathrm{SU}(n,q)\}$, $\cntr:=Z(H)$, and $G:=H/\cntr$. We will largely work in the group $H$, appealing to the following consequence of the definition of $\nc(G)$: for $A,B \in H \setminus \cntr$, the vertices $\cntr A$ and $\cntr B$ of $\nc(G)$ are adjacent if and only if $[A,B] \notin \cntr$ and $\langle A, B, \cntr \rangle < H$.

\subsection{Linear groups}
\label{subsec:nclin}

Assume that $H = \mathrm{SL}(n,q)$, with $n \ge 2$, and $q \ge 4$ if $n = 2$, so that $G$ is simple.

\begin{thm}
\label{thm:linearnc}
Let $G$ be the linear group $\mathrm{PSL}(n,q)$. Then $\diam(\nc(G)) \le 4$. Moreover, if $n = 2$, then $\diam(\nc(G)) = 2$ if $q$ is even or $q \le 9$, and  $\diam(\nc(G)) = 3$ otherwise.
\end{thm}

Let $\mathcal{U}$ be the set of one-dimensional subspaces of the natural module $V:=\mathbb{F}_q^n$ for $H$. We proceed by considering paths in $\nc(H)$ containing elements that stabilise subspaces in $\mathcal{U}$.

\begin{lem}
\label{lem:stabdist}
Let $X,Y \in \mathcal{U}$, $K \in H_X \setminus \cntr$, and $L \in H_Y \setminus \cntr$. Then $d(\cntr K,\cntr L) \le 2$.
\end{lem}

\begin{proof} We may assume that $K \ne L$. We claim that there exists a path $(U_1,\ldots,U_j)$ in $\nc(H)$, where $j \le 3$, $U_1 = K$, $U_j = L$, and $U_2 \in H_X \cap H_Y$. Since $Z(H_X) = Z(H_Y) = \cntr$ by Lemma~\ref{lem:inter1stab}\ref{inter1stab:p4}, the claim follows from Lemma~\ref{lem:propernc} if $K$ or $L$ lies in $H_X \cap H_Y$. Otherwise, Lemma~\ref{lem:inter1stab}\ref{inter1stab:p3} implies that $C_{H_X \cap H_Y}(K)$ and $C_{H_X \cap H_Y}(L)$ are proper subgroups of $H_X \cap H_Y$. Thus there exists an $M \in H_X \cap H_Y$ centralising neither $K$ nor $L$, and we can set $U_2 = M$.

Now, $[U_i,U_{i+1}] \ne 1$ for all $i \in \{1,\ldots,j-1\}$, and so $[U_i,U_{i+1}] \notin \cntr$ by Lemma~\ref{lem:inter1stab}\ref{inter1stab:p1}. Moreover, $\langle U_i, U_{i+1}, \cntr\rangle$ lies in $H_X$ or $H_Y$. Hence $(\cntr U_1,\ldots,\cntr U_j)$ is a path in $\nc(G)$.
\end{proof}

\begin{lem}
\label{lem:reduciblematnc}
Suppose that $n > 2$ or $q \not\equiv 3 \pmod 4$, and let $A \in H$ such that $\langle A \rangle$ stabilises no subspace in $\mathcal{U}$. Then there exist $X \in \mathcal{U}$ and $K \in H_X$ such that $\cntr A \sim \cntr K$.
\end{lem}

\begin{proof}
We divide the proof into two cases.

\medskip

\noindent \textbf{Case (a)}: $\langle A \rangle$ acts reducibly on $V$. It is clear that $n > 2$. The matrix $A$ is similar to a block matrix whose final row of blocks is  $\begin{matrix}(0&\cdots&0&R)\end{matrix}$, with $R$ a hypercompanion matrix and a proper submatrix of $A$ (see \cite[Theorem 8.10]{perlis}). Since $\mathrm{Aut}(G) \le \mathrm{Aut}(\nc(G))$, we may assume without loss of generality that $A$ is this block matrix.

Let $K:=I_n+E_{2,1}$. Then $K \in H_{X}$, where $X \in \mathcal{U}$ is spanned by $\begin{matrix}(1&0&\cdots&0)\end{matrix} \in V$. Each matrix in $\cntr \cup \{K\}$ can be viewed as a block matrix, with a final row $\begin{matrix}(0&\cdots&0&S)\end{matrix}$ of blocks, for some nonzero block $S$ with the same dimensions as $R$. Hence $\langle A, K, \cntr\rangle < H$.

It remains to prove that $[A, K] \notin \cntr$. For a contradiction, suppose otherwise, so that $(K^{-1})^AK = cI_n$ for some $c \in \mathbb{F}_q^\times$. Then $(K^{-1})^A = cK^{-1}$. As $K$ has eigenvalue $1$ with algebraic multiplicity $n$, while $cK^{-1}$ has eigenvalue $c$, we deduce that $[A,K] = 1$. However, $(AK)_{11} = 1$ and $(KA)_{11} = 0$, a contradiction.

\medskip

\noindent \textbf{Case (b)}: $\langle A \rangle$ acts irreducibly on $V$. By \cite[Proposition 3.6.1]{delauney}, $A$ lies in a Singer subgroup $S$ of $\mathrm{GL}(n,q)$, i.e., a cyclic subgroup of order $q^n-1$. Additionally, $N:=N_{\mathrm{GL}(n,q)}(S)$ is the group $S \sd \langle B \rangle$ of order $n(q^n-1)$, where $B \in \mathrm{GL}(n,q)$ is similar to the companion matrix $C(x^n-1)$ and satisfies $F^B = F^q$ for each $F \in S$ (see \cite[pp.~187--188]{huppertgruppen} and \cite[p.~497]{hestenes}). As $C(x^n-1)$ stabilises the subspace in $\mathcal{U}$ spanned by the all-ones vector, $B$ also stabilises some $X \in \mathcal{U}$. Note that $\det(B) = \det(C(x^n-1))$ is equal to $1$ if $q$ is even or $n$ is odd, and $-1$ otherwise. Additionally, $N$ is soluble, and so its subgroup $\langle A,B,\cntr\rangle$ is proper in $H$. Furthermore, $[A,B] = A^{-1}A^B = A^{q-1}$. Since $A^{q^2-1} \in \langle A^{q-1} \rangle$, Lemma~\ref{lem:irredmatlowdim} and our assumptions on $n$ and $q$ yield $[A,B] \notin \cntr$.

It again suffices to find a matrix $K \in H_X$ with $[A,K] \notin \cntr$ and $\langle A, K, \cntr \rangle < H$. If $n$ is odd or $q$ is even, then set $K=B$. If instead $n = 2$ and $q \equiv 1 \pmod 4$, then set $K = \alpha B$ for some square root $\alpha \in \mathbb{F}_q$ of $-1$. Finally, if $q$ is odd and $n$ is even and greater than $2$, then $B^2 \in H_X$. Additionally, $[A,B^2] = A^{q^2-1}$, which does not lie in $\cntr$ by Lemma~\ref{lem:irredmatlowdim}. Thus we set $K=B^2$.
\end{proof}

\begin{proof}[Proof of Theorem~\ref{thm:linearnc}]
Let $x,y \in G \setminus \{1\}$. We will bound $d(x,y)$. By Lemma~\ref{lem:stabdist}, $d(x,y) \le 2$ if $\langle x \rangle$ and $\langle y \rangle$ stabilise one-dimensional subspaces of $V$. Since $\diam(\nc(G)) \ge 2$, we may assume without loss of generality that $\langle x \rangle$ stabilises no such subspace. If $n \ge 3$, then Lemma~\ref{lem:reduciblematnc} shows that $x \sim s$ for some $s \in G_X$ with $X \in \mathcal{U}$. Similarly, $d(y,t) \le 1$ for some $t \in G_Y$ with $Y \in \mathcal{U}$. Hence $d(x,y) \le d(x,s)+d(s,t)+d(t,y) \le 4$, as required.

Suppose from now on that $n = 2$. Magma computations show that $\diam(\nc(G)) = 2$ when $q$ is odd and at most $9$, so we will assume that $q$ is even or at least $11$. Let $p$ be the prime dividing $q$, and let $r:=(2,q-1)$. In what follows, all information about elements and subgroups of $G$ is from \cite[Ch.~XII]{dickson}, except where stated otherwise. We first note that $|G| = q(q^2-1)/r$, and that each non-identity element of $G$ lies in either a unique cyclic subgroup of order $(q-1)/r$, a unique cyclic subgroup of order $(q+1)/r$, or a unique Sylow $p$-subgroup isomorphic to the elementary abelian group $E$ of order $q$. Moreover, all cyclic subgroups of $G$ of a given order $(q \pm 1)/r$ are conjugate. In addition, each element of order dividing $q$ or $(q-1)/r$ lies in a parabolic maximal subgroup of shape $E \sd \frac{q-1}{r}$, which is the stabiliser in $G$ of a subspace in $\mathcal{U}$.

It now follows that $|x|$ divides $(q+1)/r$. The unique cyclic subgroup of order $(q+1)/r$ containing $x$ lies in a unique maximal subgroup $L \cong D_{2(q+1)/r}$. We divide the remainder of the proof into two cases.

\medskip

\noindent \textbf{Case (a)}: $q$ is even. Suppose first that $|y|$ does not divide $q-1$. All involutions of $G$ are conjugate, as are all cyclic subgroups of order $q+1$, and so $y$ lies in a $G$-conjugate $M$ of $L$. As $Z(L) = 1$, Lemma~\ref{lem:propernc} yields $d(x,y) \le 2$ if $L = M$. If instead $L \ne M$, then $L \cap M$ contains an involution $a$ \cite[Lemma 2.3]{fritzsche}, and it is clear that $C_L(a) = \langle a \rangle = C_M(a)$. Hence $x \sim a$ and $d(a,y) \le 1$, and so $d(x,y) \le 2$.

Next, suppose that $|y|$ divides $q-1$. Then $y$ lies in a parabolic subgroup $P$ of $G$ of shape $E \sd (q-1)$. As $|P||L| = 2|G|$, there exists an involution $b \in P \cap L$. The centraliser of $b$ in $G$ is the Sylow $2$-subgroup of $P$, and thus $x \sim b \sim y$. Therefore, $\diam(\nc(G)) \le 2$.

\medskip

\noindent \textbf{Case (b)}: $q$ is odd. There exist $G$-conjugates $J$ and $K$ of $L$ with $J \cap K = 1$ \cite[Lemma 2.2]{fritzsche}. As above, $J$ is the unique maximal subgroup of $G$ containing an element $s$ of order $(q+1)/2$, and $K$ is the unique maximal subgroup containing a conjugate $t$ of $s$. Neither $s$ nor any neighbour of $s$ in $\nc(G)$ is a neighbour of $t$, and so $d(s,t) \ge 3$. Hence it remains to show that $d(x,y) \le 3$.

Suppose first that $|y|$ divides $(q+1)/2$. Notice that $Z(L) = 1$ if $q \equiv 1 \pmod 4$, else $|Z(L)| = 2$. Additionally, since all cyclic subgroups of $G$ of a given order $(q \pm 1)/2$ are conjugate, so are all involutions of $G$. It follows that there exist $G$-conjugates $R$ and $S$ of $L$ with $x \in R \setminus Z(R)$ and $y \in S \setminus Z(S)$. If $R = S$, then $d(x,y) \le 2$ by Lemma~\ref{lem:propernc}. Otherwise, as the dihedral group $R$ is generated by a pair of non-commuting involutions, there exists an involution $a \in R \setminus (Z(R) \cup S)$. In addition, Lemma~\ref{lem:simplenoncent}\ref{simplenoncent3} yields an involution $b \in S \setminus Z(S)$ with $[a,b] \ne 1$. Since $\langle a, b \rangle$ is dihedral, and since $C_R(a) = Z(R) \times \langle a \rangle$ and $C_S(b) = Z(S) \times \langle b \rangle$, we deduce that $(x,a,b,y)$ contains a path in $\nc(G)$ from $x$ to $y$. Hence $d(x,y) \le 3$.

Assume from now on that $|y|$ does not divide $(q+1)/2$, so that $y$ lies in a parabolic subgroup $P$ of shape $E \sd \frac{q-1}{2}$, with $P = G_X$ for some $X \in \mathcal{U}$. If $q \equiv 1 \pmod 4$, then Lemma~\ref{lem:reduciblematnc} yields $x \sim k$ for some $k \in G_Y$ with $Y \in \mathcal{U}$. Since $d(k,y) \le 2$ by Lemma~\ref{lem:stabdist}, we obtain $d(x,y) \le 3$.

Suppose therefore that $q \equiv 3 \pmod 4$, and let $c$ be an element of $P$ of order $(q-1)/2$. Then $\langle c \rangle$ is the pointwise stabiliser $G_{(X,Y_c)}$ for some $Y_c \in \mathcal{U} \setminus \{X\}$, and $N_c:=N_G(\langle c \rangle)$ is dihedral of order $q-1$ and equal to the setwise stabiliser $G_{\{X,Y_c\}}$. It is also clear that $c \sim f$ for each of the $(q-1)/2$ involutions $f \in N_c$, and in fact $y \sim f$ if $y \in \langle c \rangle$, since $|y| \ne 2$. In addition, $C_G(\langle c \rangle) = \langle c \rangle$, and so if $y \notin \langle c \rangle$, then $y \sim c \sim f$. Now, let $c'$ be an element of $P \setminus \langle c \rangle$ of order $|c|$, so that $Y_c \ne Y_{c'}$. Then $N_c \cap N_{c'} = G_{\{X,Y_c\}} \cap G_{\{X,Y_{c'}\}} = G_{(X,Y_c)} \cap G_{(X,Y_{c'})} = \langle c \rangle \cap \langle c' \rangle = 1$. Since there are $q$ cyclic subgroups of order $|c|$ in $P$, it follows that $y$ has distance at most two from at least $m:=q(q-1)/2$ involutions of $G$. This in fact accounts for all involutions of $G$, since the centraliser in $G$ of an involution is conjugate to $L$, which has order $|G|/m$. Hence if $|x| = 2$, then $d(x,y) \le 2$. Otherwise, $x$ is adjacent to each involution in $L \setminus Z(L)$, and so $d(x,y) \le 3$, completing the proof.
\end{proof}

Using Magma, we see that $\nc(G)$ has diameter $2$ when $G \in \{\mathrm{PSL}(3,4),\mathrm{PSL}(4,2)\}$, and diameter $3$ when $G \in \{\mathrm{PSL}(3,3),\mathrm{PSL}(3,5), \mathrm{PSL}(3,7), \mathrm{PSL}(4,3)\}$. Note that the diameter of $\nc(\mathrm{PSL}(3,4))$ was incorrectly reported as $3$ in \cite{saulthesis}.

\subsection{Unitary groups}
\label{subsec:unitarync}

Assume that $H = \mathrm{SU}(n,q)$, with $n \ge 3$, and $q \ge 3$ if $n = 3$, so that $G$ is simple. In fact, by Proposition~\ref{prop:ncsimpleeven} and Theorem~\ref{thm:simpleoddmax}, we may assume that $n$ is prime.

\begin{thm}
\label{thm:ncuniodd}
Let $G$ be the unitary group $\mathrm{PSU}(n,q)$, with $n$ an odd prime.
\begin{enumerate}[label={(\roman*)},font=\upshape]
\item $\nc(G)$ is connected with diameter at most $5$.
\item \label{ncuniodd2} Suppose that $n = 7$ and $q = 2$. Then $\nc(G)$ is connected with diameter $4$.
\end{enumerate}
\end{thm}

Let $V$ be the natural module $\mathbb{F}_{q^2}^n$ for $H = \mathrm{SU}(n,q)$, and fix a basis $\{e_1,\ldots,e_n\}$ for $V$ so that the unitary form on $V$ has Gram matrix $I_n$.

\begin{lem}
\label{lem:unithreemat}
Suppose that $n = 3$, let $\omega$ be a primitive element of $\mathbb{F}_{q^2}$, and let $\lambda:=\omega^{q-1}$. Let $B_1:=\mathrm{diag}(\lambda^{-2},\lambda,\lambda)$, $B_2:=\mathrm{diag}(\lambda,\lambda^{-2},\lambda)$, and $B_3:=\mathrm{diag}(\lambda,\lambda,\lambda^{-2})$. Additionally, let $A_1$ and $A_2$ be $H$-conjugates of $B_1$.
\begin{enumerate}[label={(\roman*)},font=\upshape]
\item \label{unithreemat1} $B_i \in H \setminus \cntr$ for each $i$.
\item \label{unithreemat2} $B_2$ and $B_3$ are each conjugate to $B_1$ in $H$.
\item \label{unithreemat3} $\langle A_1, A_2 \rangle$ stabilises a one-dimensional subspace of $V$.
\item \label{unithreemat4} If $[B_i,A_1] = 1$ for all $i$, then $A_1 \in \{B_1,B_2,B_3\}$.
\end{enumerate}
\end{lem}

\begin{proof}
Since $\lambda^q = \lambda^{-1}$ and $\det(B_i) = 1$, we see that $B_i \in H$. Additionally, $\omega^{3(q-1)} \ne 1$, and so $B_i \notin \cntr$, hence (i).  For (ii), note that the matrices
$\left( \begin{matrix}
0&1&0\\
1&0&0\\
0&0&-1\\
\end{matrix} \right)$ and $\left( \begin{matrix}
0&0&1\\
0&-1&0\\
1&0&0\\
\end{matrix} \right)$ 
lie in $H$ and conjugate $B_1$ to $B_2$ and $B_3$, respectively.

To prove (iii), we may assume that $A_1 = B_1$. Let $F \in H$ such that $A_2 = B_1^F$, and note that $B_1$ and $A_2$ act as $\lambda$ on the subspaces $E:=\langle e_2,e_3 \rangle$ and $E^F$ of $V$, respectively. Hence $\langle B_1, A_2 \rangle$ stabilises each subspace of $E \cap E^F$, which has positive dimension since $\dim(V) = 3$.

Finally, $C_H(\{B_1,B_2,B_3\})$ consists of diagonal matrices, and all $H$-conjugates of $B_1$ have the same characteristic polynomial as $B_1$. Since $B_1$, $B_2$ and $B_3$ are the only diagonal matrices in $\mathrm{GL}(n,q^2)$ with this characteristic polynomial, we obtain (iv).
\end{proof}

\begin{lem}
\label{lem:uniphid}
Suppose that $n$ is an odd prime.
\begin{enumerate}[label={(\roman*)},font=\upshape]
\item \label{uniphid1} The companion matrix $C(x^n-1)$ is an element of $H$.
\item \label{uniphid2} Let $R$ be an element of $H$ that is conjugate to $C(x^n-1)$. Then $\cntr R$ is adjacent in $\nc(G)$ to an involution of $G$.
\item \label{uniphid3} Let $R$ and $S$ be $H$-conjugates of the companion matrix $C(x^n-1)$. Then $d(\cntr R,\cntr S)$ is at most $3$ if $n = 3$, or at most $2$ otherwise.
\end{enumerate}
\end{lem}

\begin{proof}\leavevmode

\noindent (i) This holds since $C(x^n-1)^{-1} = C(x^n-1)^{T}$ and $\det(C(x^n-1)) = 1$.

\medskip

\noindent (ii) We may assume that $R = C(x^n-1)$. If $q$ is odd, then let $A$ be the direct sum of the matrices $I_{n-2}$ and $-I_2$, so that the $(1,1)$ and $(n,n)$ entries of $[R,A]$ are equal to $-1$ and $1$, respectively. If instead $q$ is even, then let $A \in H$ be the direct sum of the matrices $I_{n-2}$ and $\begin{pmatrix} 0 & 1\\1&0\\ \end{pmatrix}$, so that the $(1,1)$ entry of $[R,A]$ is equal to $0$. In either case, $[R,A] \notin \cntr$. Moreover, $R$ and $A$ are monomial, and so $\langle R, A, \cntr \rangle < H$. Thus $\cntr R \sim \cntr A$. As $A$ is an involution, so is $\cntr A$.

\medskip

\noindent (iii) We may assume that $R = C(x^n-1)$ and $d(\cntr R,\cntr S) > 1$. Then $R$ stabilises the one-dimensional subspace $X$ of $V$ spanned by the all-ones vector. Let $F \in H$ such that $S = R^F$, and let $L:=H_X$, so that $R \in L$ and $S \in L^F$. Additionally, let $L_0:=H_{\langle e_1 \rangle}$. It is easy to show that $C_{L_0}(R) = \cntr$, and thus $C_{L_0^F}(S) = \cntr$. Furthermore, as $\cntr \le L$ and $L < H$, it follows from Lemma~\ref{lem:inter1stab}\ref{inter1stab:p1} that if $[R,A] \ne 1$ for some $A \in L$, then $\cntr R \sim \cntr A$. Similarly, if $[S,D] \ne 1$ for some $D \in L^F$, then $\cntr S \sim \cntr D$. We split the remainder of the proof into three cases.

\medskip

\noindent \textbf{Case (a)}: $n \ge 7$, or $n = 5$ and $q \ne 4$. By Theorem~\ref{thm:unisubspaces}, the pointwise stabiliser $L \cap L^F \cap L_0 \cap L_0^F$ in $H$ of $\{X,X^F,\langle e_1 \rangle, \langle e_1 \rangle^F\}$ contains a non-scalar matrix $D$. The previous paragraph shows that $D$ centralises neither $R$ nor $S$, and in fact that $\cntr R \sim \cntr D \sim \cntr S$.

\medskip

\noindent \textbf{Case (b)}: $n = 5$ and $q = 4$. Straightforward computations in Magma (with a runtime of about 1.5 hours) show that $|L \cap L^F|$ is even (for all $F \in H$), while $|C_H(R)| = 625$ (and hence $|C_H(S)| = 625$). It follows from the second last paragraph that $\cntr R \sim \cntr D \sim \cntr S$ for each involution $D \in L \cap L^F$.

\medskip

\noindent \textbf{Case (c)}: $n = 3$. For each $i \in \{1,2,3\}$, define $B_i \in H \setminus \cntr$ as in Lemma~\ref{lem:unithreemat}. It is easy to check that $[R,B_i] \notin \cntr$, and since $\cntr$, $R$ and $B_i$ are monomial matrices, $\langle \cntr, R, B_i \rangle < H$. Thus $\cntr R \sim \cntr B_i$. Similarly, $\cntr S \sim \cntr B_i^F$. Recall also from Lemma~\ref{lem:unithreemat}\ref{unithreemat2}--\ref{unithreemat3} that $\langle B_i, B_1^F \rangle$ stabilises a one-dimensional subspace of $V$. Hence, as in the third last paragraph, if $[B_i, B_1^F] \ne 1$ for some $i$, then $\cntr B_i \sim \cntr B_1^F$, and so ${\cntr R \sim \cntr B_i \sim \cntr B_1^F \sim \cntr S}$. Otherwise, Lemma~\ref{lem:unithreemat}\ref{unithreemat4} yields $B_1^F = B_i$ for some $i$, and so $\cntr R \sim \cntr B_i \sim \cntr S$.
\end{proof}

\begin{proof}[Proof of Theorem~\ref{thm:ncuniodd}]\leavevmode

\noindent (i) Let $x,y \in G \setminus Z(G)$. We will show that $d(x,y) \le 5$. By Proposition~\ref{prop:isolvertnc}, each of $x$ and $y$ is non-central in some maximal subgroup. Using Lemma~\ref{lem:ncmaxsimple}, we may assume that each maximal subgroup $K$ of $G$ with $x \in K \setminus Z(K)$ has odd order. By \cite[Theorem 2]{liebecksaxl}, $K = N_H(S)/\cntr$, where $S$ is a Singer subgroup of $H$, i.e., the intersection of $H$ and a Singer subgroup of $\mathrm{GL}(n,q^2)$. It follows from \cite[p.~497 \& p.~512]{hestenes} that $N_H(S) = S \sd \langle C \rangle$ for some $H$-conjugate $C$ of $C(x^n-1)$.

Now, by Lemma~\ref{lem:uniphid}\ref{uniphid2}, $\cntr C \sim r$ for some involution $r \in G$, and so $\cntr C, r \in L \setminus Z(L)$ for some maximal subgroup $L$ of $G$ of even order. As $|C| = n$ is prime, each non-identity element of $\langle \cntr C \rangle$ lies in $L \setminus Z(L)$. Since $\cntr \le S$, it follows that $x = \cntr AC^i$ for some $A \in S \setminus \cntr$ and some $i \in \{0,1,\ldots,n-1\}$. Again using the fact that $n$ is prime, it follows from Clifford's Theorem that $\langle A \rangle$ acts irreducibly on $\mathbb{F}_{q^2}^n$. Similarly to Case (b) in the proof of Lemma~\ref{lem:reduciblematnc} (but now working over $\mathbb{F}_{q^2}$), we see that $[A,C] = A^{q^2-1}$. By Lemma~\ref{lem:irredmatlowdim}, $A^{q^4-1} \notin \cntr$, and so $[A,C] \notin \cntr$. Moreover, $\langle A, C, \cntr \rangle \le N_H(S) < H$. Therefore, $\cntr A \sim \cntr C$. It is now easy to see that $\cntr AC^i \sim \cntr C$ for all $i$, and in particular, $x \sim \cntr C$.

If $y \in M \setminus Z(M)$ for some maximal subgroup $M$ of $G$ of even order, then since $\cntr C \sim r$ and $|r| = 2$, Lemma~\ref{lem:ncmaxsimple} shows that $d(\cntr C,y) \le 4$. Hence $d(x,y) \le 5$. Otherwise, $y$ lies in a $G$-conjugate of $N_H(S)/\cntr$. By the previous paragraph, $y \sim \cntr D$ for some $H$-conjugate $D$ of $C$. Since $x \sim \cntr C$, and since $d(\cntr C,\cntr D) \le 3$ by Lemma~\ref{lem:uniphid}\ref{uniphid3}, we again obtain $d(x,y) \le 5$.

\medskip

\noindent (ii) By \cite[Remark 1.3]{intgraph} (see also \cite[Ch.~4]{saulthesis}), the intersection graph of $G$ has diameter $5$. Thus $\diam(\nc(G)) \ge 4$ by Lemma~\ref{lem:intncgraphs}. It remains to show that $d(x,y) \le 4$ for all $x, y \in G \setminus \{1\}$. In what follows, except where stated otherwise, all information about $G$ is determined using Magma, via computations involving conjugacy classes of elements and subgroups of $G$, with a runtime of about two minutes.

If $|x|,|y| \ne 43$, then each of $x$ and $y$ is adjacent in $\nc(G)$ to an involution, and Lemma~\ref{lem:ncmaxsimple} yields $d(x,y) \le 4$. Assume from now on that $|x| = 43$. Then $x$ lies in no maximal subgroup of $G$ of even order. As in the proof of (i), $x \sim t$, where $t$ is the image in $G$ of some $H$-conjugate of $C(x^7-1)$. Since $n > 3$, it follows from Lemma~\ref{lem:uniphid}\ref{uniphid3} that if $|y| = 43$, then $d(x,y) \le 4$.

Suppose finally that $|y| \ne 43$, and let $T:=\langle t \rangle$. Then there exist maximal subgroups $K$ and $L$ of $G$ such that $y \in K \setminus Z(K)$, $t \in L \setminus Z(L)$, and either:
\begin{enumerate}[label={(\Roman*)},font=\upshape]
\item $|K \cap L| > |C_G(t)|+|Z(K)|$; or
\item $Z(K) = 1$, $C_L(t) = T$, $T \cap K = 1$, and $|K|\,|L| > |G|$.
\end{enumerate}
In case (I), there exists an element $f \in K \cap L$ that centralises neither $t$ nor $K$. Hence ${t \sim f}$, and $d(f,y) \le 2$ by Lemma~\ref{lem:propernc}, yielding $d(t,y) \le 3$. In case (II), we observe that $K \cap L > 1$, and that $t \sim w$ for each non-identity $w \in K \cap L$. As $d(w,y) \le 2$ by Lemma~\ref{lem:propernc}, we again obtain $d(t,y) \le 3$. We conclude in general that $d(x,y) \le d(x,t)+d(t,y) \le 4$.
\end{proof}

Using Magma, we observe that the non-commuting, non-generating graphs of $\mathrm{PSU}(3,3)$, $\mathrm{PSU}(3,4)$ and $\mathrm{PSU}(4,2)$ have diameters $2$, $2$ and $3$, respectively.

\subsection{The non-generating graph of a non-abelian finite simple group}

\begin{proof}[Proof of Theorem~\ref{thm:nongensimple}]
Let $x,y \in G \setminus \{1\}$, and let $M_1$ and $M_2$ be maximal subgroups of $G$ containing $x$ and $y$, respectively. If $|M_1|$ and $|M_2|$ are even, then $\langle a, b \rangle$ is dihedral for involutions $a \in M_1$ and $b \in M_2$. Hence $(x,a,b,y)$ contains a path from $x$ to $y$ in $\nongen$, and $d(x,y) \le 3$. Thus we are done if every maximal subgroup of $G$ has even order.

Assume now that $|M_1|$ is odd. Since $\nc(G)$ is a spanning subgraph of $\nongen$, it follows from Theorems~\ref{thm:simpleoddmax} and~\ref{thm:ncsimple} that if $G$ is not unitary of odd prime dimension, then $\diam(\nongen)$ is at most $5$ if $G = \mathbb{M}$, and at most $4$ otherwise.

Suppose therefore that $G = \mathrm{PSU}(n,q)$, with $n$ an odd prime. As in the proof of Theorem~\ref{thm:ncuniodd}, $M_1$ is the image in $G$ of the subgroup $S \sd \langle C \rangle$ of $\mathrm{SU}(n,q)$, with $S$ a Singer subgroup of $\mathrm{SU}(n,q)$ and $C$ a conjugate of the companion matrix $C(x^n-1)$. By Lemma~\ref{lem:uniphid}\ref{uniphid2}, there exists an involution $a \in G$ such that the image $c$ of $C$ in $G$ is adjacent to $a$ in $\nc(G)$, and hence in $\nongen$.

Now, if $M_2$ contains an involution $b$, then $(x,c,a,b,y)$ contains a path from $x$ to $y$ in $\nongen$, and $d(x,y) \le 4$. Otherwise, by the previous paragraph, $y \sim c'$ for some $G$-conjugate $c'$ of $c$. Moreover, $c$ and $c'$ stabilise one-dimensional subspaces $X$ and $X'$, respectively, of $\mathbb{F}_{q^2}^n$. Since $q \ne 2$ when $n = 3$, we deduce from Theorem~\ref{thm:unisubspaces} that $G_X \cap G_{X'}$ contains an element $t \ne 1$. Hence $(x,c,t,c',y)$ contains a path from $x$ to $y$ in $\nongen$, and $d(x,y) \le 4$. Therefore, $\diam(\nongen) \le 4$.

It remains to determine lower bounds for $\diam(\nongen)$ when $G \in \{\mathbb{B}, \mathrm{PSU}(7,2), \mathbb{M}\}$. Here, \cite[Theorem 1.1 \& Remark 1.3]{intgraph} and Proposition~\ref{prop:intmonst} show that the intersection graph of $G$ is connected with diameter at least $5$, and so Lemma~\ref{lem:intncgraphs} yields $\diam(\nongen) \ge 4$.
\end{proof}

\subsection*{Acknowledgements}
This work was supported by the University of St Andrews (St Leonard's International Doctoral Fees Scholarship \& School of Mathematics and Statistics PhD Funding Scholarship) and EPSRC grant number EP/W522422/1. The author is grateful to Colva Roney-Dougal and Peter Cameron for helpful discussions regarding the original thesis on which this work is based, and to Colva and anonymous referees for their useful comments on this paper.

\bibliographystyle{abbrv}
\bibliography{Simplerefs}

\end{document}